\documentclass{amsart}

\usepackage{tikz}
\usepackage{pgfplots}
\usepackage{amsthm}
\usepackage{amsxtra}
\usepackage{amssymb}
\usepackage{graphicx}
\usepackage{comment}
\usepackage{url}
\usepackage{color}
\usepackage{amsfonts}
\usepackage{textcomp}
\usepackage{perpage,hyperref,enumerate}

\usepackage{caption}
\usepackage{subcaption}
\DeclareMathAlphabet{\mathpzc}{OT1}{pzc}{m}{it}
\usepackage{lmodern}

\usepackage{sidecap}

\DeclareMathOperator{\oph}{Op_h}
\DeclareMathOperator{\ophti}{Op_{h,\tilde{h}}}
\DeclareMathOperator{\ophtis}{Op^\Sigma_{h,\tilde{h}}}

\DeclareMathOperator*{\op}{Op}
\DeclareMathOperator*{\supp}{supp}

\DeclareMathOperator*{\Id}{Id}

\usepackage{chngcntr}

\graphicspath{{Figures/}}

\newtheorem{theorem}{Theorem}

\numberwithin{prop}{section}
\newtheorem{corol}{Corollary}
\numberwithin{corol}{section}

\newtheorem{lemma}{Lemma}
\numberwithin{lemma}{section}

\numberwithin{conjecture}{section}

{\theoremstyle{definition}
\newtheorem{defin}{Definition}
\numberwithin{defin}{section}
}
\numberwithin{figure}{section}

\newcommand{\bl}{\begin{flushleft}}
\newcommand{\el}{\end{flushleft}}
\newcommand{\br}{\begin{flushright}}
\newcommand{\ert}{\end{flushright}}
\newcommand{\bc}{\begin{center}}
\newcommand{\ec}{\end{center}}

\newcommand{\imply}{\Rightarrow}

\newcommand{\numList}{\begin{enumerate}}
\newcommand{\enumList}{\end{enumerate}}

\newcommand{\composed}{\text{\textopenbullet}}

\newcommand{\e}{\epsilon}

\newcommand{\re}{\mathbb{R}}

\newcommand{\la}{\langle}
\newcommand{\ra}{\rangle}

\newcommand{\mc}[1]{\mathcal{#1}}

\theoremstyle{remark}
\newtheorem{remark}{Remark}

\newcommand{\Cc}{C_c^\infty}
\renewcommand{\O}[1]{\mathcal{O}_{#1}}
\renewcommand{\o}[1]{\mathpzc{o}_{#1}}

\newcommand{\Ph}[2]{\Psi^{#1}_{#2}}

\newcommand{\RN}[1]{\textup{\uppercase\expandafter{\romannumeral#1}}
}

\title{The $L^2$ behavior of eigenfunctions near the glancing set}
\author{Jeffrey Galkowski}

\email{jeffrey.galkowski@stanford.edu}
\address{Mathematics Department, Stanford University, 380 Serra Mall, Stanford, 
CA 94305, USA}

\begin{document}
\begin{abstract}
Let $M$ be a compact manifold with or without boundary and $H\subset M$ be a smooth, interior hypersurface. We study the restriction of Laplace eigenfunctions solving $(-h^2\Delta_g-1)u=0$ to $H$. In particular, we study the degeneration of $u|_H$ as one microlocally approaches the glancing set by finding the optimal power $s_0$ so that $(1+h^2\Delta_H)_+^{s_0}u|_H$ remains uniformly bounded in $L^2(H)$ as $h\to 0$. Moreover, we show that this bound is saturated at every $h$-dependent scale near glancing using examples on the disk and sphere. We give an application of our estimates to quantum ergodic restriction theorems.
\end{abstract}
\maketitle

\section{Introduction}
Let $(M,g)$ be a compact Riemannian manifold with or without boundary. We consider the eigenvalue problem 
$$\begin{cases}(-\Delta_g-\lambda_j^2)u_j=0&\text{ on }M\\
\la u_j,u_k\ra=\delta_{jk}\\
Bu_j=0&\text{ on }\partial M.
\end{cases}$$
Here, $\Delta_g$ is the negative Laplacian, $\la u,v\ra$ denotes the $L^2$ inner product on $M$, and either $Bu=u$ for Dirichlet eigenvalues or $Bu=\partial_\nu u$ for Neumann eigenvalues. Our main goal is to give a precise understanding of the concentration of such eigenfunctions on hypersurfaces. We say that $H\subset M$ is an \emph{interior hypersurface} if it is a smooth embedded hypersurface with $d(H,\partial M)>0$. For convenience, we write $\lambda_j=h_j^{-1}$ and $u_j=u_{h_j}$. 

Sharp $L^p$ bounds for eigenfunctions restricted to hypersurfaces have been studied by Burq--Gerard--Tzvetkov, Hassell--Tacy, Tacy, and Tataru \cite{BGT, HTacy,T,Tat}. In particular, these works show that 
\begin{equation}
\label{e:stdEst}\|u|_{H}\|_{L^2(H)}\leq C\begin{cases}h^{-1/4}&H\text{ general}\\
h^{-1/6}&H\text{ is curved}\end{cases}
\end{equation}
where we say $H$ is \emph{curved} if it has positive definite second fundamental form. Optimal bounds for restrictions of normal derivatives of eigenfunctions of the form 
\begin{equation}
\label{eqn:neumannBound}\|h\partial_{\nu_H}u|_{H}\|_{L^2(H)}\leq C\end{equation}
were given by Christianson--Hassell--Toth and Tacy in \cite{christianson2014exterior,T14}. Heuristically, $h\partial_{\nu_H}u\sim (1+h^2\Delta_H)_+^{1/2}u$, where $\Delta_H$ denotes the (negative definite) Laplace--Beltrami operator on $H$, so the bound \eqref{eqn:neumannBound} roughly says that 
\begin{equation}
\label{eqn:weightedBound}
\|(1+h^2\Delta_H)_+^{1/2}u|_{H}\|_{L^2(H)}\leq C
\end{equation}
and in fact, the bound \eqref{eqn:weightedBound} is an easy consequence of \cite[Section 4]{christianson2014exterior}. Here, 
$$(x)_+^s=\begin{cases}x^s&x>0\\0&x\leq 0.\end{cases}$$

When $H=\partial M$, concentration questions have been addressed in Barnett--Hassell--Tacy and Hassell--Tao \cite{BHT,HTao,HTaoErr}. In particular, for respectively Dirichlet and Neumann eigenfunctions we have the sharp estimates
$$\|h\partial_\nu u|_{\partial M}\|_{L^2(\partial M)}\leq C,\quad\quad \|u|_{\partial M}\|_{L^2(\partial M)}\leq Ch^{-1/3}.$$
Moreover, in \cite{BHT} the authors show that for Neumann eigenfunctions
\begin{equation}
\label{eqn:weightedBoundBoundary}\|(1+h^2\Delta_{\partial M})_+^{1/2}u|_{\partial M}\|_{L^2(\partial M)}\leq C.
\end{equation}
The authors also show that the power $1/2$ in \eqref{eqn:weightedBoundBoundary} is optimal in the sense that there are Neumann eigenfunctions such that replacing $1/2$ by $\rho<1/2$ may result in an $L^2$ norm that is not uniformly bounded. 

\subsection{Results}

This raises the question of whether the power $1/2$ in \eqref{eqn:weightedBound} is optimal. We will see that the optimal power is $1/4$ for interior hypersurfaces. Throughout the rest of the paper, we use the notation $a+$ or $a-$ to mean that a statement holds respectively with $a$ replaced by $a+\e$ and $a-\e$ for any $\e>0$. When we use this notations, all constants may depend on the $\e$ chosen.
\begin{theorem}
\label{thm:simple}
Let $H\subset M$ be an interior hypersurface. Then if $H$ is curved or $H$ is totally geodesic 
$$\left\|(1+h^2\Delta_H)_+^{1/4+}u|_{H}\right\|_{L^2(H)}\leq C.$$
\end{theorem}
\noindent For the definition of a totally geodesic hypersurface see \eqref{e:bichar}.
Theorem \ref{thm:simple} is actually a consequence of our next theorem (together with \eqref{e:stdEst}) which applies to more general hypersurfaces. 

Before stating our next theorem, we introduce some notation for a regularization of $(1+h^2\Delta_H)_+^{s}$. Let $\chi_1,\chi_2\in C^\infty(\re)$ with $\chi_1\equiv 1$ on $[2,\infty)$, $\supp \chi_1\subset [1,\infty)$ and 
$\chi_1+\chi_2\equiv 1.$
Let 
\begin{gather} 
G_1^{\rho,s}(\sigma):=\sigma^s\chi_1\left(\frac{\sigma}{h^\rho}\right),\quad\quad G_2^{\rho,s}(\sigma):=h^{s\rho}\chi_2\left(\frac{\sigma}{h^\rho}\right)\label{e:Gdef}\\
G^{\rho,s}(\sigma):=G_1^{\rho,s}(\sigma)+G_2^{\rho,s}(\sigma).\nonumber
\end{gather}
We define
$G_{i}^{\rho,s}(1+h^2\Delta_H)$ using the functional calculus.

\begin{theorem}
\label{thm:laplace}
Let $H\subset M$ be an interior hypersurface. Then
\begin{gather*} \left\|\Big[G_1^{2/3-,1/4}(1+h^2\Delta_H)\Big]u|_{H}\right\|_{L^2(H)}\leq C(\log h^{-1})^{-1/2}\\
\left\|\Big[G_1^{2/3,1/4+}(1+h^2\Delta_H)\Big]u|_{H}\right\|_{L^2(H)}\leq C.
\end{gather*}
and
\begin{gather*} \left\|\Big[G_1^{2/3-,-1/4}(1+h^2\Delta_H)\Big]h\partial_{\nu_H}  u|_{H}\right\|_{L^2(H)}\leq C(\log h^{-1})^{-1/2}\\
\left\|\Big[G_1^{2/3,-1/4+}(1+h^2\Delta_H)\Big]h\partial_{\nu_H}  u|_{H}\right\|_{L^2(H)}\leq C.
\end{gather*}
If $H$ is nowhere tangent to the geodesic flow to infinite order,
\begin{gather*} 
\left\|\Big[G_1^{2/3-,1/4}(1+h^2\Delta_H)\Big]u|_{H}\right\|_{L^2(H)}+\left\|\Big[G_1^{2/3,1/4+}(1+h^2\Delta_H)\Big]u|_{H}\right\|_{L^2(H)}\leq C,\\
\left\|\Big[G_1^{2/3-,1/4}(1+h^2\Delta_H)\Big]h\partial_{\nu_H}  u|_{H}\right\|_{L^2(H)}+\left\|\Big[G_1^{2/3,1/4+}(1+h^2\Delta_H)\Big]h\partial_{\nu_H} u|_{H}\right\|_{L^2(H)}\leq C,
\end{gather*}
Moreover, if $H$ is totally geodesic, then 
\begin{gather*}
\left\|\Big[G_1^{1-,1/4}(1+h^2\Delta_H)\Big]u|_{H}\right\|_{L^2(H)}+\left\|\Big[G_1^{1,1/4+}(1+h^2\Delta_H)\Big]u|_{H}\right\|_{L^2(H)}\leq C,\\
\left\|\Big[G_1^{1-,-1/4}(1+h^2\Delta_H)\Big]h\partial_{\nu_H}  u|_{H}\right\|_{L^2(H)}+\left\|\Big[G_1^{1,-1/4+}(1+h^2\Delta_H)\Big]h\partial_{\nu_H}  u|_{H}\right\|_{L^2(H)}\leq C.
\end{gather*}
\end{theorem}
The power 1/4 in Theorem \ref{thm:laplace} is optimal in the sense that replacing $1/4$ by $s<1/4$ may result in an $L^2$ norm that is not uniformly bounded as $h\to 0$. Moreover, the power $1/4$ is optimal at every scale. In particular, letting $\mu=2/3$ if $H$ is not totally geodesic and $1$ otherwise, for each $0\leq \rho_1<\rho_2< \mu$, we give examples $(H,u_h)$ so that 
$$\|G_1^{\mu,s}(1+h^2\Delta_H)1_{[1-h^{\rho_1},1-h^{\rho_2}]}(-h^2\Delta_H)u_h|_{H}\|_{L^2(H)}\geq Ch^{\rho_2(s-1/4)}.$$
Since $1/4$ in Theorem \ref{thm:laplace} is strictly less than the power $1/2$ in \eqref{eqn:weightedBoundBoundary}, just as with unweighted $L^2$ bounds, weighted $L^2$ bounds are less singular on interior hypersurfaces than on boundaries.
\begin{remark}
We conjecture that for general $H$, 
$$\|G^{1,1/4}(1+h^2\Delta_H)u|_H\|_{L^2(H)}\leq C,$$
but our techniques showing the equivalence of microlocalization on $H$ and microlocalization on $M$ fail at scale $h^{2/3}$ unless $H$ is totally bicharacteristic. 
\end{remark}

More generally, we consider a semiclassical pseudodifferential operator $P$ with real principal symbol, $p(x,\xi).$ 
Let 
$$\Sigma_{x_0}:=\{\xi\mid p(x_0,\xi)=0\}\subset T^*_{x_0}M\}.$$
We assume that
\begin{equation}
\label{eqn:assume}
\begin{gathered} 
p(x_0,\xi_0)=0\quad\imply \quad\partial_\xi p(x_0,\xi_0)\neq 0,\quad\quad\qquad\qquad \lim_{|\xi|_g\to \infty}|p(x,\xi)|=\infty\\
\Sigma_{x_0} \text{ has positive definite second fundamental form and is connected for each }x_0.
\end{gathered}
\end{equation}
\begin{remark}
The assumption that $\Sigma_{x_0}$ be connected is not essential, but we make it to simplify the presentation.
\end{remark}
Furthermore, we say that \emph{$H$ is curved} if the projection of the bicharacteristic flow is at most simply tangent to $H$. That is, for any defining function $r$ for $H$, 
\begin{equation}
\label{e:curved}p(x_0,\xi_0)=r(x_0)=H_pr(x_0,\xi_0)=0\quad \quad \imply \quad\quad H_p^2r(x_0,\xi_0)\neq 0\end{equation}
where $H_p$ denotes the Hamiltonian vector field of $p$. We say that \emph{$H_p$ is tangent to $H$ to infinite order at $(x_0,\xi_0)$} if for all $k>0$,
$$p(x_0,\xi_0)=r(x_0)=H_p^kr(x_0,\xi_0)=0.$$
Finally, let $\Phi_t:T^*M\to T^*M$ be the Hamiltonian flow of $p$ given by $\Phi_t(x,\xi)=\exp(tH_p)(x,\xi)$. We say that \emph{$H$ is totally bicharacteristic near $(x_0,\xi_0)$ if} 
\begin{equation}\label{e:bichar}p(x_0,\xi_0)=r(x_0)=H_pr(x_0,\xi_0)=0\quad\quad\imply \quad\quad  \Phi_t^*r(x_0,\xi_0)\equiv 0\text{ for }t\text{ in a neighborhood of }0.\end{equation}

 Let $\pi:T^*_{H}M\to T^*H$ be given by orthogonal projection and $\nu$ denote a fixed normal to $H$.
Let 
$$\Sigma:=\{p=0\},\quad\quad \mc{G}:=\Sigma\cap\{\partial_\nu p=0\},\quad\quad\Sigma_0:=\pi(\Sigma),\quad\quad \mc{G}_0:=\pi(\mc{G})=\partial \Sigma_0.$$ 
(For the fact that under \eqref{eqn:assume}, $\partial\Sigma_0=\mc{G}_0$, see Section \ref{sec:structure}.)
\begin{defin}
We say that $b\in S^m(T^*H;\re)$ \emph{defines} $\mc{G}_0$ if $b$ is a defining function for $\mc{G}_0$, $b>0$ on $\Sigma_0\setminus \mc{G}_0$, and $|b|>c\la \xi'\ra^m>0$ on $|\xi'|_g\geq M$.  
\end{defin}
Let $\gamma_H:u\mapsto u|_H$ denote the restriction operator.
Theorem \ref{thm:laplace} is then an easy consequence of the following theorem.
\begin{theorem}
\label{thm:general}
Suppose that $H\subset M$ is an interior hypersurface and that $P$ has principle symbol $p$ satisfying \eqref{eqn:assume}. Suppose that $b\in S^m(T^*H)$ defines $\mc{G}_0$. Then there exists $\e>0$ small enough so that for $\psi\in \mc{S}$ with $\psi(0)=1$ and $\supp\hat{\psi}\subset[-\e,\e]$ we have
\begin{gather*} 
\left\|G_1^{2/3-,1/4}(b(x',hD_{x'}))\gamma_H\psi\left(\frac{P}{h}\right)\right\|_{L^2(M)\to L^2(H)}\leq C(\log h^{-1})^{-1/2}\\
\left\|G_1^{2/3,1/4+}(b(x',hD_{x'}))\gamma_H\psi\left(\frac{P}{h}\right)\right\|_{L^2(M)\to L^2(H)}\leq C,\\
\left\|G_1^{2/3-,-1/4}(b(x',hD_{x'}))\gamma_H\partial_\nu p(x,hD)\psi\left(\frac{P}{h}\right)\right\|_{L^2(M)\to L^2(H)}\leq C(\log h^{-1})^{-1/2},\\
\left\|G_1^{2/3,-1/4+}(b(x',hD_{x'}))\gamma_H\partial_\nu p(x,hD)\psi\left(\frac{P}{h}\right)\right\|_{L^2(M)\to L^2(H)}\leq C.
\end{gather*}
If $H$ is nowhere tangent to $H_p$ to infinite order, then 
\begin{gather*} 
\left\|G_1^{2/3-,1/4}(b(x',hD_{x'}))\gamma_H\psi\left(\frac{P}{h}\right)\right\|_{L^2(M)\to L^2(H)}+\left\|G_1^{2/3,1/4+}(b(x',hD_{x'}))\gamma_H\psi\left(\frac{P}{h}\right)\right\|_{L^2(M)\to L^2(H)}\leq C,\\
\left\|G_1^{2/3-,-1/4}(b(x',hD_{x'}))\gamma_H\partial_\nu p(x,hD)\psi\left(\frac{P}{h}\right)\right\|_{L^2(M)\to L^2(H)}\leq C,\\
\left\|G_1^{2/3,-1/4+}(b(x',hD_{x'}))\gamma_H\partial_\nu p(x,hD)\psi\left(\frac{P}{h}\right)\right\|_{L^2(M)\to L^2(H)}\leq C,
\end{gather*} 
and if $H$ is totally bicharacteristic, then 
\begin{gather*} 
\left\|G_1^{1-,1/4}(b(x',hD_{x'}))\gamma_H\psi\left(\frac{P}{h}\right)\right\|_{L^2(M)\to L^2(H)}+\left\|G_1^{1,1/4+}(b(x',hD_{x'}))\gamma_H\psi\left(\frac{P}{h}\right)\right\|_{L^2(M)\to L^2(H)}\leq C,\\
\left\|G_1^{1-,-1/4}(b(x',hD_{x'}))\gamma_H\partial_\nu p(x,hD)\psi\left(\frac{P}{h}\right)\right\|_{L^2(M)\to L^2(H)}\leq C,\\
\left\|G_1^{1,-1/4+}(b(x',hD_{x'}))\gamma_H\partial_\nu p(x,hD)\psi\left(\frac{P}{h}\right)\right\|_{L^2(M)\to L^2(H)}\leq C.
\end{gather*} 
Furthermore, the power $1/4$ is sharp in the sense for for any power less than $1/4$, examples exists where these operators are not uniformly bounded in $h$.
\end{theorem}
Combining Theorem \ref{thm:general} with the the analog of the estimates \eqref{e:stdEst} for quasimodes gives
\begin{corol}
Let $H$ be an interior hypersurface. Then if $H$ is curved, 
$$\left\|G^{2/3,1/4+}(b(x',hD_{x'}))\gamma_H\psi\left(\frac{P}{h}\right)\right\|_{L^2(M)\to L^2(H)}\leq C,$$ 
and if $H$ is totally bicharacteristic, 
$$\left\|G^{1,1/4+}(b(x',hD_{x'}))\gamma_H\psi\left(\frac{P}{h}\right)\right\|_{L^2(M)\to L^2(H)}\leq C.$$ 
\end{corol}

Finally, we give an application of our estimates to quantum ergodic restriction theorems. We say that a sequence of eigenfunctions of the Laplacian, $u_h$, is \emph{quantum ergodic} if for all $A\in \Ph{}{}(M)$, 
$$\la Au_h,u_h\ra\underset{h\to 0}{\longrightarrow}\frac{1}{\mu_L(S^*M)}\int_{S^*M}\sigma(A)(x,\xi)d\mu_L$$
where $\mu_L$ is the Liouville measure on $S^*M$. By the now classical quantum ergodicity theorem of Shnirleman \cite{Snir}, Colin de Verdi{\`e}re \cite{dever}, Zelditch \cite{zeld}, and Zelditch--Zworski \cite{ZeZw}, if the (broken) geodesic flow on $M$ is ergodic, than there is a full density subsequence of eigenfunctions which is quantum ergodic. 

More recently, there has been interest in quantum ergodic properties of restrictions of eigenfunctions. Dyatlov--Zworski \cite{DyZw}, and Toth--Zelditch \cite{TZ1,TZ2} showed that, under an asymmetry condition on $H$, there is a further full density subsequence of $u_h$, such that for $A\in \Ph{}{}(H)$,
\begin{equation}
\label{e:asymmetry}
\la A u_h|_H,u_h|_H\ra \to \frac{2}{\mu_L(S^*M)}\int_{B^*H}\sigma(A)(x,\xi')(1-|\xi'|_g^2)^{-1/2}dxd\xi'.\end{equation}

Moreover, Christianson--Toth--Zelditch \cite{CTZ} show that without the need to make an additional asymmetry condition or to take a further full density subsequence 
\begin{equation}
\label{e:cauchy}\la Ah\partial_{\nu_H} u_h|_H,h\partial_{\nu_H} u_h|_H\ra +\la (1+h^2\Delta_H)Au_h|_H,u_h|_H\ra \to \frac{4}{\mu_L(S^*H)}\int_{B^*H}\sigma(A)\sqrt{1-|\xi'|_g^2}dxd\xi'.\end{equation}
One should notice that there is an extra factor of $(1+h^2\Delta_H)$ in the second term of \eqref{e:cauchy} when compared to \eqref{e:asymmetry}. This is due to the fact that (even quantum ergodic) eigenfunctions may have bad concentration properties near trajectories tangent to the hypersurface $H$. However, Theorem \ref{thm:general} gives us uniform control over how bad this concentration may be and as a consequence, we can reduce the number of factors of $(1+h^2\Delta_H)$ required.
\begin{theorem}
\label{thm:QER}
Suppose that $u_h$ is quantum ergodic and $A\in \Ph{}{}(H)$. Then for all $s<1/2$, 
\begin{multline*} 
\la G_1^{2/3,-s}(1+h^2\Delta_H)Ah\partial_{\nu_H} u_h,h\partial_{\nu_H} u_h\ra +\la G_1^{2/3,1-s}(1+h^2\Delta_H) u_h|_H,u_h|_H\ra\\\ \to \frac{4}{\mu_L(S^*M)}\int_{B^*H}\sigma(A)(x,\xi')(1-|\xi'|_g^2)^{1/2-s}dxd\xi'.\end{multline*}
\end{theorem}

\subsection{Outline of the proof of Theorem \ref{thm:general}}
To prove Theorem \ref{thm:general}, we start by proving estimates on restrictions of normal frequency bands of $\psi(P/h)$. In particular, let $\nu$ be a fixed conormal to $H$. Then we use \cite{T14} to obtain estimates on
\begin{equation}
\label{eqn:est1}\left\|\gamma_H\chi\left(\frac{\partial_\nu p(x,hD)}{\tilde{h}}\right)\psi\left(\frac{P}{h}\right)\right\|_{L^2(M)\to L^2(H)}.
\end{equation}
Observe that
$$P\psi\left(\frac{P}{h}\right)=\O{L^2\to L^2}(h)$$ 
and since $\{\xi\mid p(x_0,\xi)=0\}$ is compact for all $x_0$, there exist $\chi \in \Cc(\re)$ such that 
$$(1-\chi(|hD|_g))\psi\left(\frac{P}{h}\right)=\O{C^\infty}(h^\infty).$$ 
Therefore, to obtain the estimates on \eqref{eqn:est1}, we need only prove estimates for quasimodes, $u$ such that $u$ is compactly microlocalized, $\|u\|_{L^2(M)}\leq 1$, and $Pu=\O{L^2}(h)$. 

Our next task is to give restriction estimates on normal frequency bands of $u$. In particular, let $\chi \in \Cc(\re)$ with $\supp \chi \subset [1/2,4]$. Using \cite[Proposition 1.1]{T14}, we show that for $\tilde{h}\gg h$,
\begin{equation}
\label{eqn:cutEst}\gamma_H\chi\left(\frac{\partial_\nu p(x,hD)}{\tilde{h}^{1/2}}\right)u=\O{L^2(H)}(\tilde{h}^{-1/4}).\end{equation}

To deduce Theorem \ref{thm:general} from \eqref{eqn:cutEst}, we need to show that for a quasimode of $P$, microlocalization at scale $\tilde{h}^{1/2}$ away from $\mc{G}$ in the ambient manifold passes to $\tilde{h}$ microlocalization away from the $\mc{G}_0$ after composition with $\gamma_H$. Because of the square root singularity in $\pi:\Sigma\to \Sigma_0$ near $\mc{G}_0$, we need to use the second microlocal calculus from \cite{SjoZwDist,SjZwFrac}. More precisely, we show that for $\chi_\nu\in \Cc(\re)$ with $\supp \chi_\nu \subset [1,2]$, there exists $\chi\in \Cc(\re)$ such that 
\begin{equation}\label{eqn:cut1}\left(1-\chi\left(\frac{b(x',hD_{x'})}{\tilde{h}}\right)\right)\gamma_H\chi_\nu\left(\frac{|\partial_\nu p(x,hD)|}{\tilde{h}^{1/2}}\right)\psi\left(\frac{P}{h}\right)\end{equation}
is negligible (see Figure \ref{f:micStruc} for a schematic view of the various microsupports). This will only be possible when $\tilde{h}\gg h^{2/3}$ unless $H$ is totally bicharacteristic. 
Finally, to complete the proof of Theorem \ref{thm:general}, we use an almost orthogonality argument. 

\begin{figure}[htbp]
\includegraphics[width=.8\textwidth]{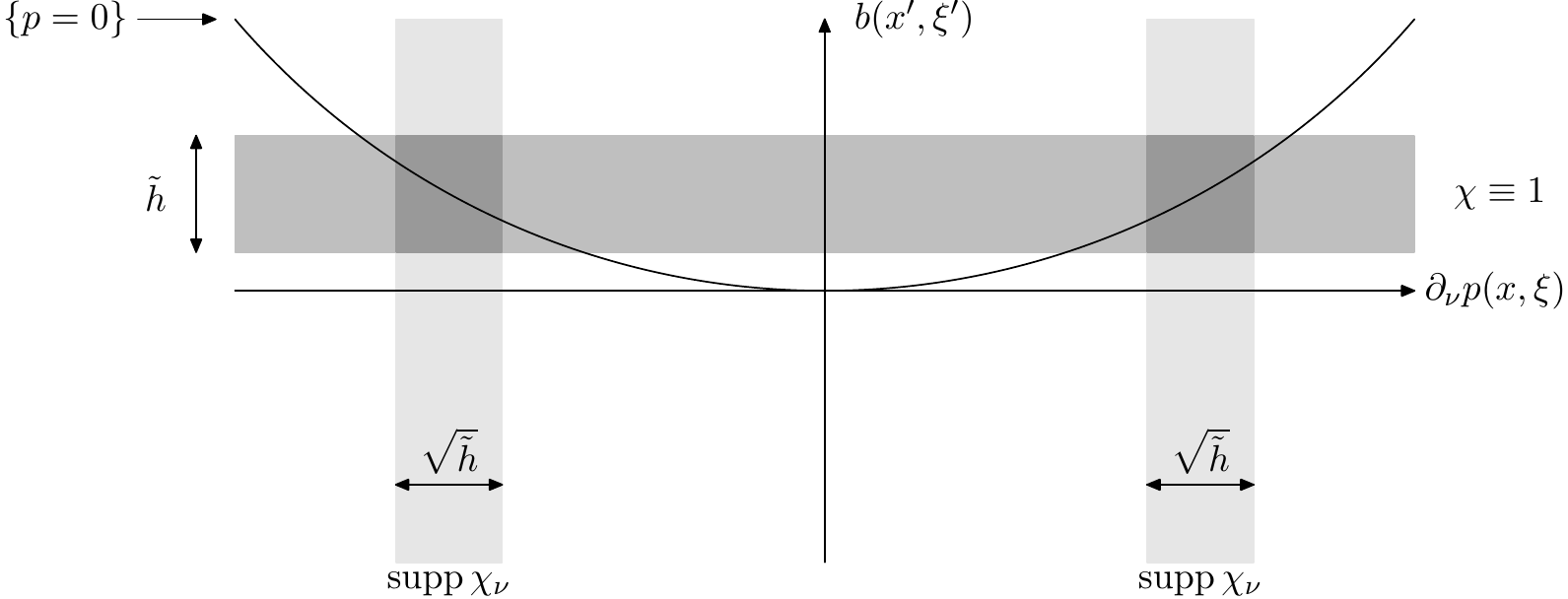}
\caption{\label{f:micStruc}The figure shows the supports of the various pseudodifferential cutoffs in \eqref{eqn:cut1}. }
\end{figure}

\subsection{Organization of the paper}
In Section \ref{sec:prelim}, we review some facts from the second microlocal calculus of \cite{SjoZwDist, SjZwFrac}.
In Section \ref{sec:normal}, we adapt Tacy's methods \cite{T14} for our purposes. Next, in Section \ref{sec:structure}, we examine the geometry of $\mc{G}_0$, $\mc{G}$, $\Sigma$, and $\Sigma_0$ for general Hamiltonians $p$. Then, in Section \ref{sec:microlocalize}, we prove that small scale microlocalization in $T^*M|_H$ away from $\mc{G}$ passes to small scale microlocalization in $T^*H$ away from $\mc{G}_0$.  Next, in Section \ref{sec:final}, we complete the proof of the main theorem. In Section \ref{sec:optimal}, we show that the power $1/4$ cannot be improved. Finally, in Section \ref{sec:QER}, we prove Theorem \ref{thm:QER} as an application for our estimates.

\vspace{1cm}
\noindent
{\sc Acknowledgemnts.} The author would like to thank Suresh Eswarathasan for the many stimulating discussions that started this project and for his careful reading of an earlier version of this paper. Thanks also to John Toth, Andras Vasy, and Maciej Zworski for valuable suggestions. The author is grateful to the National Science Foundation for support under the Mathematical Sciences Postdoctoral Research Fellowship  DMS-1502661.

\section{Second microlocalization at a hypersurface}
\label{sec:prelim}
\label{sec:second}
In this section, we review the necessary results from the second microlocal calculus associated to a hypersurface from \cite{SjoZwDist, SjZwFrac} where one can find more details. Throughout, let $(M,g)$ be a compact Riemannian manifold of dimension $d$ with $T^*M$ its cotangent bundle. 
\subsection{The basic calculus}
Here, we collect some facts from the standard semiclassical calculus (see \cite[Chapter 4]{EZB}, \cite[Chapter 7]{d-s} for more details). We first introduce symbol classes. For $\delta\leq 1/2$,
$$S^{m}_\delta(T^*M):=\{a\in C^\infty(T^*M)\mid |\partial_x^\alpha\partial_{\xi}^\beta a|\leq C_{\alpha\beta}h^{-\delta(|\alpha|+|\beta|)}\la \xi\ra^{m-|\beta|}\}$$
where $\la \xi\ra:=(1+|\xi|_g^2)^{1/2}.$
We also define an $\tilde{h}$ class of symbols when $\delta=1/2$. 
$$S^{m}_{1/2,\tilde{h}}(T^*M):=\{a\in C^\infty(T^*M)\mid |\partial_x^\alpha\partial_{\xi}^\beta a|\leq C_{\alpha\beta}h^{-\frac{1}{2}(|\alpha|+|\beta|)}\tilde{h}^{|\alpha|+|\beta|}\la \xi\ra^{m-|\beta|}\}.$$

Then, the corresponding $h$-Weyl pseudodifferential operators are operators that in local coordinates have Schwartz kernels of the form
$$K_a(x,y)=\frac{1}{(2\pi h)^{d}}\int e^{\frac{i}{h}\la x-y,\xi\ra}a\left(\frac{x+y}{2},\xi;h\right)d\xi.$$
Here, the integral is defined as an oscillatory integral (see \cite[Section 3.6]{EZB}). We write $\oph_{\re^d}(a)$ for the operator with Schwartz kernel $K_a$. 

Then we have the following lemma in local coordinates \cite[Theorems 4.11,4.12,9.5]{EZB}
\begin{lemma}
\label{lem:pseudCompose}
For $0\leq \delta\leq 1/2$, $a\in S^{m_1}_{\delta}(T^*\re^d)$, and $b\in S^{m_2}_{\delta}(T^*\re^d)$
$$\oph_{\re^d}(a)\oph_{\re^d}(b)=\oph_{\re^d}(c)$$
where
$$c=\left.e^{ih\sigma(D_x,D_\xi,D_y,D_\eta)/2}a(x,\xi)b(y,\eta)\right|_{\substack{x=y\\\xi=\eta}}\in S^{m_1+m_2}_\delta(T^*\re^d).$$
Moreover, for $\delta<1/2$, or $a\in S^{m_1}_{1/2,\tilde{h}}$, $b\in S^{m_2}_{1/2,\tilde{h}},$ $c$ has an asymptotic expansion
$$c\sim \sum_j \frac{i^jh^j}{j!2^j}\left.\left[(\sigma(D_x,D_\xi,D_y,D_\eta))^j(a(x,\xi)b(y,\eta))\right]\right|_{\substack{x=y\\\xi=\eta}}.$$
In particular, if $\supp a\cap \supp b=\emptyset$, then
$$c=\O{}( \tilde{h}^{\infty}\la \xi\ra^{-\infty}),\quad\quad c=\O{}(h^\infty\la \xi\ra^{-\infty}).$$
respectively for $a,b\in S^*_{1/2,\tilde{h}}$ and $a, b\in S^*_\delta$ for $\delta<1/2$. 
\end{lemma}
Moreover, we have the following boundedness lemma \cite[Theorems 4.23, 8.10]{EZB}
\begin{lemma}
\label{lem:L2Bound}
There exists a constant $M$, such that for all $s$, $a\in S^m_{1/2}(T^*\re^d)$
$$\|\oph_{\re^d}(a)u\|_{H_h^s}\leq C\left(\sum_{|\alpha|\leq Md}h^{|\alpha|/2}\sup |\partial^\alpha a(x,\xi)\la \xi\ra^{-m}|\right) \|u\|_{H_h^{s+m}}$$
where 
$$\|u\|_{H_h^s}:=\|\la hD\ra^su\|_{L^2}.$$
\end{lemma}

For $0\leq \delta\leq 1/2$, let $\Ph{m}{\delta}(M)$ denote the class of pseudodifferential operators with symbol in $S^m_\delta(T^*M)$ (see for example \cite[Appendix E]{ZwScat} \cite[Chapter 14]{EZB}) and define a global quantization procedure and symbol map
$$\oph(a):S^m_{\delta}\to \Ph{m}{\delta}(M),\quad \sigma:\Ph{m}{\delta}(M)\to S^m_{\delta,\tilde{h}}(T^*M)/h^{1-2\delta}\tilde{h}S^{m-1}_{\delta,\tilde{h}}(T^*M)$$ so that for $b$ real, $\oph(b)$ is symmetric, and
\begin{gather*} 
\oph(1)=\Id,\quad \sigma\circ \oph=\pi
\end{gather*} 
where $\pi$ is the natural projection map.
When it is convenient, we will sometimes write $a(x,hD)$ for $\oph(a)$.

\subsection{Second microlocal operators along a hypersurface, $\Sigma$}
We now review calculus of second microlocal pseudodifferential operators associated to a hypersurface (see \cite{SjoZwDist, SjZwFrac} for a more complete treatment). Let $\Sigma\subset T^*M$ be a compact embedded hypersurface with $M$ a manifold of dimension $d$. For $0\leq \delta \leq 1$, we say that 
$a\in S^{k_1,k_2}_{\Sigma,\delta,\tilde{h}}(T^*M)$ if 
\begin{equation}
\label{e:symbRequirement}
\begin{cases}
\text{near }\Sigma:\, V_1\dots V_{l_1}W_1\dots W_{l_2}a=\O{}(h^{-\delta l_2}\tilde{h}^{l_2}\la h^{-\delta}\tilde{h}d(\Sigma,\cdot)\ra ^{k_1}),\\
\text{where }V_1\dots V_{l_1}\text{ are tangent to }\Sigma\\
\text{and } W_1\dots W_{l_2}\text{ are any vector fields}\\
\text{away from }\Sigma:\, \partial_x^\alpha\partial_\xi^\beta a(x,\xi)=\O{}(\la h^{-\delta }\tilde{h}\ra^{k_1}\la \xi\ra^{k_2-|\beta|}).
\end{cases}
\end{equation}
where we take $\tilde{h}=1$ if $\delta<1$ and write $S^{k_1,k_2}_{\Sigma,\delta,\tilde{h}}(T^*M)$ and $\tilde{h}$ small with $h$ chosen small enough depending on $\tilde{h}$ for $\delta=1$. 
To define a class of operators associated to these symbol classes, we proceed locally and put $\Sigma$ into the normal form $\Sigma_0=\{\xi_1=0\}$. Let $a=a(x,\xi,\lambda;h)$, $\lambda=\tilde{h}h^{-\delta}\xi_1$ be defined near $(0,0)$ so that near $\Sigma_0$ \eqref{e:symbRequirement} becomes
\begin{equation}
\label{e:localSymbol}\partial_x^\alpha\partial_\xi^\beta\partial_\lambda^k a=\la \lambda\ra^{k_1-k}.
\end{equation}
If \eqref{e:localSymbol} holds, we write
$$a=\widetilde{\O{}}(\la \lambda\ra^{k_1}).$$
For such $a$, we define the quantization
$$\widetilde{\ophti}(a)u(x):=\frac{1}{(2\pi h)^d}\int a\left(\frac{x+y}{2},\xi,\tilde{h}h^{-\delta}\xi_1;h\right)e^{\frac{i}{h}\la x-y,\xi\ra} u(y)dyd\xi.$$
Then, using Lemma \ref{lem:pseudCompose}, we see that 
\begin{lemma}
\label{lem:pseudCompose}
For $0\leq \delta\leq 1$, $a=\widetilde{\O{}}(\la \lambda\ra^{k_1})$, and $b=\widetilde{\O{}}(\la \lambda\ra^{k_2})$. Then, 
$$\widetilde{\ophti}(a)\widetilde{\ophti}(b)=\widetilde{\ophti}(c)$$
where
$$c=\left.e^{ih\sigma(D_x,\tilde{h}h^{-\delta}D_{\lambda}+D_{\xi_1},D_{\xi'},D_y,\tilde{h}h^{-\delta}D_{\omega}+D_{\eta_1},D_{\eta'})/2}a(x,\xi,\lambda)b(y,\eta,\omega)\right|_{\substack{x=y\\\xi=\eta\\\lambda=\omega}}=\widetilde{\O{}}(\la \lambda\ra^{k_1+k_2}\ra).$$
Moreover, for $c$ has an asymptotic expansion
$$c\sim \sum_j \frac{i^jh^j}{j!2^j}\left.\left[(\sigma(D_x,\tilde{h}h^{-\delta}D_{\lambda}+D_{\xi_1},D_{\xi'},D_y,\tilde{h}h^{-\delta}D_{\omega}+D_{\eta_1},D_{\eta'}))^j(a(x,\xi)b(y,\eta))\right]\right|_{\substack{x=y\\\xi=\eta\\\lambda=\omega}}.$$
In particular, if $\supp a\cap \supp b=\emptyset$, then
$$c=\O{}( \tilde{h}^{\infty}(h^{1-\delta})^\infty\la \lambda\ra^{-\infty} ).$$
\end{lemma}

For an operator $\widetilde{\ophti}(a)$, we define its principle symbol by the equivalence class of $a$ in 
$$\widetilde{\O{}}(\la \lambda\ra^{k_1})/\widetilde{\O{}}(\tilde{h}h^{1-\delta}\la \lambda\ra^{k_1-1}).$$
The $L^2$ boundedness for second microlocal operators follows easily from Lemma \ref{lem:L2Bound}.
\begin{lemma}
\label{lem:secondL2Bound}
For $a=\widetilde{\O{}}(\la \lambda\ra^{k_1})$ with bounded support in $\xi_1$, there exists $C>0$ such that 
$$\|\widetilde{\ophti}(a)u\|_{L^2}\leq Ch^{-\delta\max(k_1,0)}\|u\|_{L^2}.$$
\end{lemma}
\begin{proof}
We have
\begin{align*}
\|\widetilde\ophti(a)\|_{L^2\to L^2}
&\leq \|\op_1((a(y,h\eta,h^{1-\delta}\tilde{h}\xi_1))\|_{L^2\to L^2}\\
&\leq C\sum_{|\alpha|\leq Md}\|\partial^\alpha a(y,h\eta,h^{1-\delta}\tilde{h}\xi_1)\|_{L^\infty}
\end{align*}
where the last line follows by Lemma \ref{lem:L2Bound} applied with $h=1$. 
\end{proof}

The class 
$$\Ph{k_1}{\Sigma_0,\delta}(\re^d):=\{\widetilde{\ophti}(a)\mid a=\widetilde{\O{}}(\la \lambda\ra^{k_1})\}$$
is invariant under conjugation by $h$-Fourier integral operators preserving $\Sigma_0$.

We say that $A=B$ microlocally on an opens set $U\subset T^*M$ if for any $a,a'\in \Cc(T^*M)$ supported in a small enough neighborhood of $U$, 
$$\oph(a)(A-B)\oph(b)=\O{\mc{D}'\to C^\infty}(\tilde{h}^{\infty}(h^{1-\delta})^\infty).$$ 

Now, we define the global class of operators $\Ph{k_1,k_2}{\Sigma,\delta,\tilde{h}}(M)$ by saying $A\in\Ph{k_1,k_2}{\Sigma,\delta,\tilde{h}}(M)$ if and only if for any point $p\in \Sigma$ and elliptic $h$-FIO, $U:C^\infty(M)\to C^\infty(\re^d)$ quantizing a symplectomorphism, $\kappa$, with 
$$\kappa(p)=(0,0),\quad\text{and}\quad\kappa (\Sigma\cap V)\subset\Sigma_0$$
where $V$ is some neighborhood of $p$, microlocally near $(0,0)$, 
$$UAU^{-1}=\widetilde{\ophti}(\widetilde{\O{}}(\la \lambda\ra^{k_1}))$$ 
and for any point $p\notin\Sigma$, $A\in (h^{-\delta}\tilde{h})^{k_1}\Ph{k_2}{}(M)$ microlocally near $p$. 

For $a\in S^{k_1,k_2}_{\Sigma,\delta,\tilde{h}}(T^*M)$, we define a quantization procedure using the normal form. Let $\psi\in \Cc(T^*M)$ have $\psi\equiv 1$ on $\{d(p,\Sigma)\leq \e\}$ and $\supp \psi\subset\{d(p,\Sigma)\leq 2\e\}$ for some $\e>0$ to be chosen small enough. We then find a finite cover $W_j$ of $\supp \psi$ such that there exists a neighborhood $V$ of $(0,0)\in T^*\re^d$ such that for each $j$ there is a symplectomorphism $\kappa_j$  
$$\kappa_j:V\to W_j,\quad \kappa_j(V\cap \Sigma_0)= \Sigma\cap W_j.$$ 
Then choose elliptic $h-$FIO's $U_j$ quantizing $\kappa_j$ defined microlocally in a neighborhood of $V\times W_j$. Let $\varphi_j$ be a partition of unity on $\{d(p,\Sigma)\leq 2\e\}$ subordinate to $W_j$ and define $a_j$ as the unique symbol of the form $a_j=a_j(x,\xi',\lambda;h)$ such that 
$$(a_j)|_{\lambda=\tilde{h}h^{-\delta}\xi_1}=(\psi \phi_j a)\circ \kappa_j,$$
and define
$$\ophtis(a)=\ophti((1-\psi)a)+\sum_jU_j^{-1}\widetilde{\ophti}(a_j)U_j.$$
By adjusting the $U_j$, we may arrange so that 
$$\ophtis(1)=\Id.$$
Then we have the following lemma \cite[Proposition 4.1]{SjoZwDist}. 
\begin{lemma}
\label{lem:secondPseud}
There exists maps 
$$\ophtis:S^{k_1,k_2}_{\Sigma,\delta,\tilde{h}}(T^*M)\to \Ph{k_1,k_2}{\Sigma,\delta,\tilde{h}}(M)$$
 and 
 $$\sigma_{\Sigma}:\Ph{k_1,k_2}{\Sigma,\delta,\tilde{h}}(M)\to S^{k_1,k_2}_{\Sigma,\delta,\tilde{h}}(T^*M)/h^{1-\delta}\tilde{h}S^{k_1-1,k_2-1}_{\Sigma,\delta,\tilde{h}}(T^*M).$$
 Such that 
 \begin{gather*}
 \sigma_\Sigma(AB)=\sigma_\Sigma(A)\sigma_\Sigma(B),\\
 0\to h^{1-\delta}\tilde{h}\Ph{k_1-1,k_2-1}{\Sigma,\delta,\tilde{h}}(M)\to \Ph{k_1,k_2}{\Sigma,\delta,\tilde{h}}(M)\underset{\sigma_\Sigma}{\longrightarrow}S^{k_1,k_2}_{\Sigma,\delta,\tilde{h}}(T^*M)/h^{1-\delta}\tilde{h}S^{k_1-1,k_2-1}_{\Sigma,\delta,\tilde{h}}(M)
 \end{gather*}
 is a short exact sequence,
 $$\sigma_{\Sigma}\circ\ophtis:S^{k_1,k_2}_{\Sigma,\delta,\tilde{h}}(T^*M)\to S^{k_1,k_2}_{\Sigma,\delta,\tilde{h}}(T^*M)/h^{1-\delta}\tilde{h}S^{k_1-1,k_2-1}_{\Sigma,\delta,\tilde{h}}(T^*M)$$
 is the natural projection map and if $a\in S^{k_1,k_2}_{\Sigma,\delta,\tilde{h}}(T^*M)$ is supported away from $\Sigma$, then $\ophtis(a)\in h^{-\delta k_1}\Ph{k_2}{}(M).$ Finally, if $\supp a\cap \supp b=\emptyset$, then 
 $$\ophtis(a)\ophtis(b)=\O{\mc{D}'\to C^\infty}((h^{1-\delta}\tilde{h})^{\infty}).$$
\end{lemma}
\begin{remark}
When $\delta=1$, we will take the residual class to be operators which are $\O{\mc{D}'\to C^\infty}(\tilde{h}^\infty)$ microlocally near $\Sigma$. The residual class actually has addition properties which are often convenient (see \cite[Section 5.4]{SjZwFrac}), but this will be enough for our purposes.
\end{remark}

Our last task will be to show that the operators $G_i^{\rho,s}(b(x,hD))$ are pseudodifferential operators in the second microlocal calculus. 
Let $b(x,\xi)\in S^m(T^*M;\re)$ with 
$$|b(x,\xi)|\geq c\la \xi\ra^m>0,\quad\quad |\xi|_g\geq M,$$
$\psi(t)\in \Cc(\re)$ and $\chi\in \Cc(\re)$ with $\chi \equiv1$ near 0.
Notice that under these assumptions, $\oph(b)$ is self adjoint with domain $H^m$. 

We consider $\psi(\oph(b)\tilde{h}h^{-\delta})$ microlocally near a point $(x_0,\xi_0)$. We have
\begin{align*} 
\psi (\oph(b)\tilde{h}h^{-\delta})&=\frac{h^\delta}{2\pi h\tilde{h}}\int e^{\frac{i}{h}t\oph(b)}\hat{\psi}(\tilde{h}^{-1}h^{\delta-1}t) dt\\
&=\frac{h^\delta}{2\pi h\tilde{h}}\int e^{\frac{i}{h}t\oph(b)}\hat{\psi}(\tilde{h}^{-1}h^{\delta-1}t)\chi(t) dt\\
&\quad\quad\qquad+\frac{h^\delta}{2\pi h\tilde{h}}\int e^{\frac{i}{h}t\oph(b)}\hat{\psi}(\tilde{h}^{-1}h^{\delta-1}t)(1-\chi(t)) dt\\
&=\frac{h^\delta}{2\pi h\tilde{h}}\int e^{\frac{i}{h}t\oph(b)}\hat{\psi}(\tilde{h}^{-1}h^{\delta-1}t)\chi(t) dt+\O{\mc{D}'\to C^\infty}((\tilde{h}h^{1-\delta})^{\infty})\\
&=\frac{h^\delta}{(2\pi h)^{d+1}\tilde{h}}\int e^{\frac{i}{h}(\varphi(t,x,\theta)-\la y, \theta\ra+\la h^\delta\tilde{h}^{-1}t,\tau\ra)}a(t,x,\theta)\psi(\tau) \chi(t)dtd\tau d\theta\\
&\quad\quad\quad\quad+\O{\mc{D}'\to C^\infty}((\tilde{h}h^{1-\delta})^{\infty})
\end{align*}
where 
$$\partial_t\varphi=b(x,\partial_x\varphi),\quad \varphi(0,x,\theta)=\la x,\theta\ra,\quad a(0,x,\theta)=1+\O{}(h).$$
Then, performing stationary phase in the $t,\tau$ variables gives 
$$\frac{1}{(2\pi h)^{d}}\int e^{\frac{i}{h}\la x- y, \theta\ra}a_1(x,\theta)d\theta$$
where 
$$a_1(x,\theta)\sim \psi(\tilde{h}h^{-\delta}b(x,\theta))+\sum_{j=1}^\infty (h^{1-\delta}\tilde{h})^jL_{2j}(a\psi)\Big|_{\substack{t=0\\\tau=\tilde{h}h^{-\delta}b(x,\theta)}}$$
and $L_{2j}$ is a differential operator of order $2j$ in $t$ and $s$. 
Now, changing coordinates so that $b(x,\xi)=\xi_1$ and using a microlocal partition of unity, proves the following lemma
\begin{lemma}
\label{lem:funcCalculus}
Let $b(x,\xi)\in S^m(T^*M;\re)$ define $\Sigma$ and have 
$$|b(x,\xi)|\geq c\la \xi\ra^m,\quad\text{on }|\xi|_g>M.$$
Then for $\psi\in \Cc(\re)$, 
$$\psi (\oph(b)\tilde{h}h^{-\delta})\in\Ph{-\infty,-\infty}{\Sigma,\delta,\tilde{h}}$$
and 
$$\sigma_{\Sigma}(\psi (\oph(b)\tilde{h}h^{-\delta}))=\psi(b(x,\xi)\tilde{h}h^{-\delta}).$$
\end{lemma}
Now, let $\psi,\psi_0\in \Cc(\re)$ with $\psi\equiv 1$ on $[1,2]$ and $\supp \psi\subset [1/2,4]$ such that 
$$\psi_0(x)+\sum_{j=1}^\infty\psi_j(x)\equiv 1,\qquad \psi_j(x)=\psi(2^{-j}x).$$ 
Then, 
$$G_i^{\rho,s}(b(x,hD))=G_i^{\rho,s}(b(x,hD))\left(\psi_0(b(x,hD)h^{-\rho})+\sum_{j=1}^\infty\psi(2^{-j}b(x,hD)h^{-\rho})\right).$$
Lemma \ref{lem:funcCalculus} implies that if $\rho<1$, 
\begin{gather*} 
G_1^{\rho,s}(b(x,hD))\psi_j(b(x,hD)h^{-\rho})\in h^{\rho s}2^{js}\Ph{-\infty,-\infty}{\Sigma,\rho,1},\\
G_2^{\rho,s}(b(x,hD))\psi_j(b(x,hD)h^{-\rho})\in h^{\rho s}\Ph{-\infty,-\infty}{\Sigma,\rho,1}.
\end{gather*}
Then, the orthogonality of $\psi_j(b(x,hD)h^{-\rho})$ and $\psi_k(b(x,hD)h^{-\rho})$ for $|j-k|>2$ implies that 
$$G_1^{\rho,s}(b(x,hD))\in h^{\min(s,0)\rho}\Ph{s,-\infty}{\Sigma,\rho,1},\quad G_2^{\rho,s}(b(x,hD))\in h^{\rho s}\Ph{0,-\infty}{\Sigma,\rho,1}$$
In addition, if $\rho=1$, then 
\begin{gather*} 
G_1^{1,s}(b(x,hD))\psi_j(b(x,hD)h^{-1})\in h^{\rho s}2^{js}\Ph{-\infty,-\infty}{\Sigma,1,2^{-j}},\\
G_2^{1,s}(b(x,hD))\psi_j(b(x,hD)h^{-1})\in h^{\rho s}\Ph{-\infty,-\infty}{\Sigma,1,2^{-j}}.
\end{gather*} 
and hence for $J>0$, 
\begin{gather*} 
G_1^{1,s}(b(x,hD))\in h^{\min(s,0)}\Ph{s,-\infty}{\Sigma,1,2^{-J}}+\O{L^2\to L^2}(h^{ s}),\\
 G_2^{1,s}(b(x,hD))\in h^{ s}\Ph{0,-\infty}{\Sigma,1,2^{-J}}+\O{L^2\to L^2}(h^{s}).
 \end{gather*}

\section{Estimates on normal frequency bands}
\label{sec:normal}
We start by giving a quantitative estimate on the restriction of quasimodes when microlocalized at a certain scale from the glancing set. 
We say that $u$ is \emph{compactly microlocalized} if there exists $\chi \in \Cc(\re)$ such that 
$$u=\chi(|hD|_g)u+\O{C^\infty}(h^\infty).$$
We say that $u$ is a \emph{quasimode for $P$} if 
$$\|Pu\|_{L^2(M)}=\O{}(h)\|u\|_{L^2(M)}.$$

\begin{lemma}
\label{lem:normal}
Let $u$ be compactly microlocalized and let $\tilde{h}\geq h^{1/2}$.  Suppose that $\chi\in \Cc(\re)$ has support in $[1,4]$.
Then,
$$\left\|\gamma_H\chi\left(\frac{|\partial_\nu p(x,hD)|}{\tilde{h}}\right)u\right\|_{L^2(H)}\leq C\tilde{h}^{-1/2}(\|u\|_{L^2(M)}+Ch^{-1}\|Pu\|_{L^2(M)}).$$
In particular, if $u$ is a quasimode for $P$, then 
$$\left\|\gamma_H\chi\left(\frac{|\partial_\nu p(x,hD)|}{\tilde{h}}\right)u\right\|_{L^2(H)}\leq C\tilde{h}^{-1/2}\|u\|_{L^2(M)}.$$
\end{lemma}
\begin{proof}
We deduce the lemma from the work of Tacy \cite[Proposition 1.1]{T14}, which we recall here
\begin{lemma}[\cite{T14}]
\label{lem:tacyLem}
Let $\zeta\in \Cc(\re)$ have $\supp \zeta\subset [1,4]$. Then for $\tilde{h}\geq h^{1/2}$, 
$$\|\gamma_H(\partial_\nu p)_{\tilde{h},\zeta}(x,hD) u\|\leq C\tilde{h}^{1/2}(\|u\|_{L^2(M)}+h^{-1}\|Pu\|_{L^2(M)})$$
where 
$$(\partial_\nu p)_{\tilde{h},\zeta}(x,\xi)=\zeta(\tilde{h}^{-1}|\partial_\nu p(x,\xi)|)\partial_\nu p(x,\xi).$$
\end{lemma}
Write 
$$\tilde{\chi}(x)=\frac{1}{x}\chi(x).$$
Then, $\tilde{\chi}\in \Cc(\re)$ with $\supp \chi\subset[1,4]$ and
$$\tilde{h}^{-1}(\partial_\nu p)_{\tilde{h},\tilde{\chi}}=\chi(\tilde{h}^{-1}|\partial_\nu p(x,\xi)|).$$
Therefore, Lemma \ref{lem:tacyLem} implies
$$\left\|\gamma_H\chi\left(\frac{|\partial_\nu p(x,hD)|}{\tilde{h}}\right) u\right\|_{L^2(H)}\leq C\tilde{h}^{-1/2}(\|u\|_{L^2(M)}+h^{-1}\|Pu\|_{L^2(M)})$$
as desired.
\end{proof}

\section{The structure of $\mc{G}_0$ and $\Sigma_0$}
\label{sec:structure}
Our goal for this section is to show that near the glancing set $\mc{G}_0$, general Hamiltonians, $p$, have roughly the same structure as the Laplacian, $p=|\xi|_g^2-1$. We will do this using the Malgrange preparation theorem together with similar ideas to those used in \cite{KTZ,MelroseGlanceSurf} to put Hamiltonians in a normal form. We first recall this structure for the Laplacian.
\subsection{Structure of $\mc{G}_0$ and $\Sigma_0$ for the Laplacian} In this section, we work in Fermi normal coordinates. That is, $H=\{x_d=0\}$ and 
$$-h^2\Delta_g=(hD_{x_d})^2+R(x',hD_{x'})+2x_dQ(x_d,x',hD_{x'})+hr(x,hD_x)$$
where $R(x',\xi')=|\xi'|_{g_1}^2$ where $g_1$ is the metric induce on $H$ from $M$, $Q$ is a quadratic function of $\xi'$ such that $Q(0,\cdot,\cdot)$ is the symbol of the second fundamental form of $H$ and $r\in S^1(T^*M)$. In these coordinates,
$$p=\sigma(-h^2\Delta_g-1)=\xi_d^2+R(x',\xi')+2x_dQ(x_d,x',\xi')-1,\quad\quad\partial_\nu p(x,\xi)=2\xi_d.$$
Therefore, 
\begin{align*} 
\Sigma&=\{(x',x_d,\xi',\xi_d)\mid \xi_d^2+R(x',\xi')+2x_dQ(x_d,x',\xi')=1\},&\Sigma_0&=\{(x',\xi')\mid R(x',\xi')\leq 1\},\\
\mc{G}&=\{(x',x_d,\xi',0)\mid R(x',\xi')+2x_dQ(x_d,x',\xi')=1\},& \mc{G}_0&=\{(x',\xi')\mid R(x',\xi')=1\}.
\end{align*}
In particular, notice that $\mc{G}_0=\partial \Sigma_0$, $1-R(x,\xi')$ defines $\mc{G}_0$, and
$$\Sigma=\left\{(\partial_\nu p(x,\xi))^2=\frac{1}{4}R(x,\xi')\right\}.$$
We will show that these three facts continue to hold for a general Hamiltonian $p$ and $b$ defining $\mc{G}$. 

\subsection{Structure of $\mc{G}_0$ and $\Sigma_0$ for general Hamiltonians} 
We will show that $\partial\Sigma_0=\mc{G}_0$ and examine the structure of $\Sigma$ near $\mc{G}$.
Choose coordinates so that $H=\{x_d=0\}$. We start by considering 
$$\Sigma_0\setminus \mc{G}_0=\{(x',\xi')\in T^*H\mid \text{for all }\xi_d\text{ either }p(x',0,\xi',\xi_d)\neq 0\text{ or }\partial_{\xi_d}p(x',0,\xi',\xi_d)\neq 0\}.$$
Consider a point $(x_0',\xi_0')\in \Sigma_0\setminus \mc{G}_0$. Then either $p(x_0',0,\xi_0',\xi_d)\neq 0$ for all $\xi_d$ or there exists $\xi_d$ such that with $(x_0,\xi_0)=(x_0',0,\xi_0',\xi_d)$, $p(x_0,\xi_0)=0$ and $\partial_{\xi_d}p(x_0,\xi_0)\neq 0$. In the first case $(x_0',\xi_0')\notin \Sigma_0$. Therefore, we need only consider the second case.

In the second case, by the implicit function theorem near $(x_0,\xi_0)$, 
$$p(x,\xi)=e(x,\xi)(\xi_d-a(x,\xi'))$$
with $|e(x,\xi)|>c>0$. Therefore, there exists a neighborhood, $U$ of $(x_0',\xi_0')\in T^*H$ such that $U\subset \Sigma_0$. Hence, $(x_0',\xi_0')$ is in the interior of $\Sigma_0$. 

Now, consider a point $(x_0',\xi_0')\in \mc{G}_0$. Then there exists $\xi_d$ such that $p(x_0',0,\xi_0',\xi_d)=0$ and $\partial_{\xi_d}p(x_0',0,\xi_0',\xi_d)=0$. Let $(x_0,\xi_0)=(x_0',0,\xi_0',\xi_d)$. By assumption $\partial_\xi p\neq 0$ on $\Sigma$. Therefore, we may assume that $\partial_{\xi_1}p(x_0,\xi_0)\neq 0$, $\partial_{\xi''}p=0$ where $\xi=(\xi_1,\xi'')$. By the implicit function theorem, near $(x_0,\xi_0)$, with $x=(x_1,x'')$, $\xi=(\xi_1,\xi'')$,
\begin{equation}
\label{eqn:pform}p(x,\xi)=e(x,\xi)(\xi_1-a(x_1,x'',\xi''))\end{equation}
with $|e(x,\xi)|>c>0$. Now, $\partial_{\xi_d}p(x_0,\xi_0)=0$ implies that $\partial_{\xi_d}a(x_0,\xi''_0)=0$. By \eqref{eqn:assume}, $\Sigma_{x_0}$
has positive definite second fundamental form at $\xi_0$. Therefore, since $\partial_{\xi''}p(x_0,\xi_0)=0$, $\partial_{\xi_d}^2p(x_0,\xi_0)\neq 0$ and hence $\partial_{\xi_d}^2a(x_0,\xi_0'')\neq 0$.

Now, we assume without loss that $ \partial_{\xi_d}^2p(x_0,\xi_0)>0$, otherwise take $p\mapsto-p$. Then by the Malgrange preparation theorem 
\begin{equation}\label{eqn:pform1} p(x,\xi)=e(x,\xi)((\xi_d-a_0(x,\xi'))^2-a_1(x,\xi'))\end{equation}
where $ e>c>0$. Now, since $\partial_{\xi_1}p(x_0,\xi_0)\neq 0$, $\partial_{\xi_1}a_1(x_0,\xi_0')\neq 0$ and hence by the implicit function theorem, there exists $|e_1(x,\xi')|>c>0$, $a_2(x,\xi''')$ where $\xi=(\xi_1,\xi''',\xi_d)$ such that 
$$a_1(x,\xi')=e_1(x,\xi')(\xi_1-a_2(x,\xi''')).$$
We will assume again that $\partial_{\xi_1}p(x_0,\xi_0)>0$ ( here we can write $x_1\mapsto -x_1$ if necessary) without loss so that $e_1>c>0$. 
Therefore, 
\begin{equation}
\label{eqn:pform2}p=e(x,\xi)((\xi_d-a_0(x,\xi'))^2-e_1(x,\xi')(\xi_1-a_2(x,\xi''')))
\end{equation}
and, near $(x_0,\xi_0)$, 
$$\Sigma = \{\xi_1\geq a_2(x,\xi''')\}.$$

Now, by \eqref{eqn:assume}, for all $x_0$, $\Sigma_{x_0}$ has everywhere positive definite second fundamental form and is connected. Therefore, it is the boundary of a strictly convex set. Moreover, $\Sigma_{x_0}$ is closed since $p$ is continuous. 
Thus, for any line
$$L:=\{(x_0,\xi_0',\xi_d)\mid \xi_d\in \re\},$$
there are three options,
$\#L\cap \Sigma_{x_0}=2$, $\#L\cap \Sigma_{x_0}=0$, or $L$ is tangent to $\Sigma_{x_0}$ and $\#L\cap\Sigma_{x_0}=1$. 

Now, by \eqref{eqn:pform2} for each $(x_0,\xi_0')$ with
$$\xi_{0,1}-a_2(x_0,\xi_0''')>0,$$ 
we have found two solutions $\xi_d$ to $p(x_0,\xi_{0,1},\xi_0''',\xi_d)=0$ and hence there are no other solutions. 

Next, by \eqref{eqn:pform2} we see that if $\xi_{0,1}-a_2(x_0,\xi_0''')=0$, then $\partial_{\xi_d}p=0$ on $\Sigma_{x_0}\cap L$ and hence $L$ is tangent to $\Sigma_{x_0}$ and there is one point in the intersection, $(x_0,\xi_0)$. Moreover, $\partial_{\xi_1}p(x_0,\xi_0)\neq 0$. Therefore, there is a hyperplane $A$, supporting $\Sigma_{x_0}$ so that $\partial_{\xi_1}$ is transverse to $A$ at $(x_0,\xi_0)$, so for 
$$\xi_{0,1}-a_2(x_0,\xi_0''')<0,$$
there are no solutions $\xi_d$. Together, this implies that $\partial\Sigma_0=\mc{G}_0$. 

Moreover, for any defining function $b$ of $\mc{G}_0$, near $(x_0,\xi_0)$, there exists $\pm e_5>c>0$ such that 
$$b=e_5(x,\xi)(\xi_1-a_2(x,\xi'''))$$
and hence for some $e_6>c>0$,
\begin{equation}
\label{eqn:bdfForm}
\Sigma=\{(\xi_d-a_0(x,\xi'))^2-e_6(x,\xi')b(x,\xi')=0\}
\end{equation}

Summarizing, we have
\begin{lemma}
\label{lem:bdfForm}
Let $p$ satisfy \eqref{eqn:assume} and $H\subset M$ a smooth hypersurface with defining function $r$, Then $\mc{G}_0=\partial\Sigma_0$ and for any $b$ defining $\mc{G}$, there exist $\e>0$, a function $e\in C^\infty(T^*H)$, $c>0$ such that $e>c>0$ and for $\nu$ the dual variable to $r$, (i.e. conormal to $H$) 
$$\Sigma\cap \{|r|\leq \e\}=\{(\partial_\nu p)^2=eb\}.$$
\end{lemma}

\section{Microlocalization}
\label{sec:microlocalize}
We now prove a lemma that allows us to pass from microlocalization in $\partial_\nu p$ on $M$ to microlocalization on $H$. (See Figure \ref{f:micStruc} for a schematic of the various microlocalizers in the following lemma.)
\begin{lemma}
\label{lem:microlocalize}
Let $\psi\in \mc{S}$ such that $\supp \hat{\psi}\subset [-\e,\e]$, $\psi(0)=1$, $\chi_\nu\in \Cc(\re)$ with $\supp \chi \subset [1,4]$ and $\chi \equiv 1 $ on $[2,3]$, $Q\in \Ph{1}{}(M)$ with 
$$|\sigma(Q)|\leq |\partial_\nu p(x,\xi)|^\mu,\quad\quad\text{on }|p|\leq 1$$
and $b$ define $\mc{G}$. 
Then there exists $0<a_1<a_2$ so that for $\tilde{h}\geq h^{2/3}$, and $\chi_\nu \in \Cc(\re)$ with $\supp \chi_\nu \subset [a_1,a_2]$, 
$$\left(1-\chi\left(\frac{b(x',hD_{x'})}{ \tilde{h}}\right)\right)\gamma_H\chi_\nu\left(\frac{|\partial_\nu p(x,hD)|}{\tilde{h}^{1/2}}\right)Q\psi\left(\frac{P}{h}\right)=\O{L^2\to L^2}(h^{-}\tilde{h}^{-1/4+\mu/2}(h\tilde{h}^{-3/2})^\infty).$$
\end{lemma}
\begin{proof}

Recall that by Lemma \ref{lem:bdfForm}, there exists $e>c>0$ such that 
\begin{equation}
\label{eqn:sigForm}\Sigma=\{p=0\}=\{(\partial_\nu p)^2=eb\}.
\end{equation}
Without loss of generality, we assume that $e\equiv 1$ since otherwise we can simply absorb $e$ into $b$ and adjust $a_1,a_2$ appropriately. To prove the Lemma, we will use the normal form for second microlocal operators twice. Once before restriction to $H$ and once after. By doing this and using \eqref{eqn:sigForm}, we reduce $\Sigma$ to the form $\{\xi_d^2-\eta_1=0\}$ where we are able to easily analyze the necessary integration by parts.

We may use a partition of unity in a neighborhood of $H$ to reduce to a single coordinate chart.
First, write 
$$\psi\left(\frac{P}{h}\right)=\frac{1}{2\pi}\int e^{\frac{i}{h}tP}\hat{\psi}(t)dt.$$
Then, by \cite[Theorem 10.4]{EZB}, for a pseudodifferential operator, $A_0$ with wavefront set in a small enough neighborhood of a point $(x_0,\xi_0)\in T^*M$
$$\psi\left(\frac{P}{h}\right)A_0=\frac{1}{2\pi(2\pi h)^d}\int e^{\frac{i}{h}(\varphi(t,x,\theta)-\la y,\theta\ra)}a(t,x,y,\theta)\hat{\psi}(t)dtd\theta+\O{\Ph{-\infty}{}}(h^\infty)$$
where 
\begin{equation}
\label{eqn:relation}\partial_t\varphi(t,x,\theta)=p(x,\partial_x\varphi),\quad\quad \varphi(0,x,\theta)=\la x,\theta\ra.
\end{equation} 

Next, for $A_1=\oph(a_1)$ and $\supp(a_1)$ in a small neighborhood of $(x_0',\xi_0')$, by the definition of second microlocal operators, we have for $T_1$ unitary and quantizing a symplectomorphism, $\kappa_1$ such that $\kappa_1^*(b)=\eta_1$, the kernel of
\begin{multline*} 
T_1^{-1}A _1\left(1-\chi\left(\frac{b(x',hD_{x'})}{ \tilde{h}}\right)\right)=Ch^{-M}\int e^{\frac{i}{h}(\la y-y_1',\eta'\ra+\Phi_1(y_1',y_2',\theta_1)}a_1(y,\eta,y_1',y_2',\theta_1,\eta_1/\tilde{h})dy_1'd\eta' d\theta_1\\+\O{L^2\to C^\infty}((h\tilde{h}^{-1})^\infty)
\end{multline*}
with $a_1$ supported in $\eta_1/\tilde{h}\in \supp(1- \chi).$
Similarly, for $T_2$ unitary quantizing a symplectomorphism, $\kappa_2$ such that $\kappa_2^*(\nu )=\xi_d$ and $A_2=\oph(a_2)$ with $\supp(a_2)$ in a small neighborhood of $(x_0,\xi_0)$
\begin{multline*} 
T_2A_2\chi_\nu\left(\frac{\partial_\nu p(x,hD)}{\tilde{h}^{1/2}}\right)T_2^{-1}=Ch^{-M}\int e^{\frac{i}{h}(\la y-y_1,\xi\ra+\Phi_1(y_1,y_2,\theta_1)}a_1(y,\xi,y_1,y_2,\theta_1,\eta_d\tilde{h}^{-1/2})dy_1d\xi d\theta_1\\+\O{L^2\to C^\infty}(h^\infty).
\end{multline*}

Thus, modulo negligible terms, 
\begin{multline} 
\label{eqn:toEstimate}
T_1^{-1}A_1 \left(1-\chi\left(\frac{b(x',hD_{x'})}{ \tilde{h}}\right)\right)\gamma_HA_2\chi_\nu\left(\frac{|\partial_\nu p(x,hD)|}{\tilde{h}^{1/2}}\right)Q\psi\left(\frac{P}{h}\right)A_0\\
Ch^{-M-N}\tilde{h}^{\mu/2}\int e^{\frac{i}{h}(\la x-y_1',\eta'\ra+\Phi_1(y_1',y_2',\theta_1)-y_{2,d}\zeta_d-\Phi_2(y_3,y_2,\theta_2)+\la y_3-y_4,\xi\ra+\la y_4-y_5,\omega\ra+\Phi_2(y_5,y_6,\theta_3) +\varphi(t,y_5,\theta_4)-\la y,\theta_4\ra)}\\\tilde{a}(x,y_1',\eta',y_1',y_2',\theta_1,y_2,y_3,\theta_2,y_4,\theta_3,t,y_5,\theta_4,y_6,\omega,\xi_d\tilde{h}^{-1/2},\eta_1\tilde{h}^{-1})\hat{\psi}(t)\\dy_1'd\eta'd\theta_1dy_2d\zeta_ddy_3d\theta_2dy_4d\xi dy_5d\theta_3dt d\theta_4dy_6d\omega\end{multline}
with 
\begin{equation}
\label{e:supp1}\supp \tilde{a}\subset\{\eta_1\tilde{h}^{-1}\in \supp(1-\chi),\,\xi_d\tilde{h}^{-1/2}\in \supp \chi_\nu\}
\end{equation}

Now, let $\Psi$ denote the phase function in the above integral. Then,
\begin{equation}
\label{eqn:diff}
\begin{aligned} 
d_{y_1'}\Psi&=d_{y_1'}\Phi_1-\eta'&d_{y_2'}\Psi&=d_{y_2'}\Phi_1-d_{y_2'}\Phi_2&d_{y_{2,d}}\Psi&=-\zeta_d-d_{y_{2,d}}\Phi_2
&d_{\zeta_d}\Psi&=-y_{2,d}\\
d_{y_3}\Psi&=-d_{y_3}\Phi_2+\xi&d_{y_5}\Psi&=-\omega+d_{y_5}\Phi_2&
d_{y_6}\Psi&=d_{y_6}\Phi_2+d_{y_6}\varphi&d_t\Psi&=d_t\varphi\\
d_{\theta_1}\Psi&=d_{\theta_1}\Phi_1&d_{\theta_2}\Psi&=-d_{\theta_2}\Phi_2&d_{\theta_3}\Psi&=d_{\theta_3}\Phi_2&d_{\theta_4}\Psi&=d_{\theta_4}\varphi-y\\
d_{\xi}\Psi&=y_3-y_4&d_\omega\Psi&=y_5-y_6&dy_4\Psi&=\omega-\xi
\end{aligned}
\end{equation}
In particular, these equations imply 
\begin{equation*}
\begin{aligned}
p(y_5,\partial_{y_6}\varphi)&=0& \kappa_2(y_6,\partial_{y_6}\varphi)&=(y_4,\xi)\\
\kappa_2(y_2',0,-d_{y_2'}\Phi_2(y_4,y_2',0,\theta_2),\zeta_d)&=(y_4,\xi)&
\kappa_1(y_2',-d_{y_2'}\Phi_2(y_4,y_2',0,\theta_2))&=(y_1',\eta')
\end{aligned}
\end{equation*}
Putting this together with \eqref{eqn:relation} and letting $d$ denote the differential in all variables listed in \eqref{eqn:diff} and using the definition of $\kappa_1$, and $\kappa_2$, we have that 
$$d\Psi=0\quad\imply\quad \eta_1=\xi_d^2.$$

Together, this implies that for some $\gamma_i$ smooth and independent of $h$,
$$\xi_d^2-\eta_1=\sum_{i}\gamma_id_i\Psi$$
where $d_i$ runs over the variables in \eqref{eqn:diff}.

Now, we integrate by parts. In particular, we write 
$$\frac{h}{i}\frac{\sum_{i}\gamma_i d_i\Psi}{\xi_d^2-\eta_1}e^{\frac{i}{h}\Psi}=e^{\frac{i}{h}\Psi}.$$ 
Now, for $a_1=5,a_2=7$, in the support of the integrand, $5\tilde{h}\leq \eta_1\leq7\tilde{h}$ and $|\xi_d|\notin \tilde{h}^{1/2}[4,9].$ Therefore, the denominator is larger than $\tilde{h}.$
 Integration by parts in all variables except $\xi_d$ does not cause any difficulty. However, integration by parts in $\xi_d$ requires closer analysis. 
The $\xi_d$ derivative can fall on $\tilde{a}$, producing $h\tilde{h}^{-3/2}$ or it can fall on the denominator. In this case,
$$\partial_{\xi_d}\frac{1}{\xi_d^2-\eta_1}=\frac{2\xi_d}{(\xi_d^2-\eta_1)^2}.$$

 Suppose that $\xi_d^2\leq C\tilde{h}$, then the numerator is bounded by $\tilde{h}^{1/2}$ and hence the overall bound is $\tilde{h}^{-3/2}$. Furthermore, if $\xi_d^2\geq C\tilde{h}$, then we also have the bound $\tilde{h}^{-3/2}$. For higher order derivatives, the derivative can fall on $\tilde{a}$, $(\xi_d^2-\eta_1)^{-1}$, or $\xi_d$. We have already seen that the first two cases result in terms of size $\tilde{h}^{-3/2}$. For the last case, observe that we replace 
 $$\frac{\xi_d}{(\xi_d^2-\eta_1)^2}\to \frac{1}{(\xi_d^2-\eta_1)^3}$$
 and hence replace a factor which we bounded by $\tilde{h}^{-3/2}$ with one bounded by $\tilde{h}^{-3}$. 
 Thus, after each integration by parts, we gain $h\tilde{h}^{-3/2}$ and the integrand is bounded by 
$$C_Nh^{N}\tilde{h}^{-3N/2}.$$
Hence, there exists $M>0$ such that 
\begin{equation}
\label{eqn:poly}
\left\|\left(1-\chi\left(\frac{b(x',hD_{x'})}{ \tilde{h}}\right)\right)\gamma_H\chi_\nu\left(\frac{|\partial_\nu p(x,hD)|}{\tilde{h}^{1/2}}\right)Q\psi\left(\frac{P}{h}\right)\right\|_{L^2\to L^2}\leq C_N\tilde{h}^{\mu/2} h^{-M}\tilde{h}^{-r}h^N\tilde{h}^{-3N/2}
\end{equation}
where $r=0$.

By the results of  Lemma \ref{lem:tacyLem} applied to $Q\psi(P/h)$, the operator
$$\left\|\gamma_H\chi_\nu(\partial_\nu p(x,hD)/\tilde{h}^{1/2})Q\psi(P/h)\right\|_{L^2\to L^2}=\tilde{h}^{-1/4+\mu/2}$$
and
$$\left(1-\chi\left(b(x',hD_{x'})/\tilde{h}\right)\right)=\O{L^2\to L^2}(1).$$
Therefore, applying Cauchy Schwarz we have that
$$\left\|\left(1-\chi\left(\frac{b(x',hD_{x'})}{ \tilde{h}}\right)\right)\gamma_H\chi_\nu\left(\frac{|\partial_\nu p(x,hD)|}{\tilde{h}^{1/2}}\right)Q\psi\left(\frac{P}{h}\right)\right\|_{L^2\to L^2}\leq \tilde{h}^{\mu/2}C_{2N}^{1/2}h^{-M/2}\tilde{h}^{-(4r+1)/8}(h\tilde{h}^{-3/2})^{N}.$$
 Thus, \eqref{eqn:poly} with $M$ implies the same bound with $h^{-M}\tilde{h}^{-r}$ replaced by $h^{-\frac{M}{2}}\tilde{h}^{-\frac{(4r+1)}{8}}$ and hence \eqref{eqn:poly} is proved with $(M,r)$ replaced by $(M2^{-N}, 1/4(1-2^{-N}))$ by iterating this procedure finitely many times.
\end{proof}

\begin{lemma}
\label{lem:microlocalizeTotallyBicha}
Suppose the $H$ is totally bicharacteristic.
Let $\psi\in \mc{S}$ such that $\supp \hat{\psi}\subset [-\e,\e]$, $\psi(0)=1$,$\chi_\nu\in \Cc(\re)$ with $\supp \chi \subset [1,4]$ and $\chi \equiv 1 $ on $[2,3]$, $Q\in \Ph{1}{}(M)$
with 
$$|\sigma(Q)|\leq |\partial_\nu p(x,\xi)|,\quad\quad\text{on }|p|\leq 1,$$
 and $b$ define $\mc{G}$. 
Then there exists $0<a_1<a_2$ so that for $\tilde{h}\geq h$, and $\chi_\nu \in \Cc(\re)$ with $\supp \chi_\nu \subset [a_1,a_2]$, 
$$\left(1-\chi\left(\frac{b(x',hD_{x'})}{ \tilde{h}}\right)\right)\gamma_H\chi_\nu\left(\frac{|\partial_\nu p(x,hD)|}{\tilde{h}^{1/2}}\right)Q\psi\left(\frac{P}{h}\right)=\O{L^2\to L^2}(h^{-}\tilde{h}^{-1/4+\mu/2}(h\tilde{h}^{-1})^\infty)$$
\end{lemma}
\begin{proof}
We modify the proof of Lemma \ref{lem:microlocalize} in the case that $H$ is totally bicharacteristic. Consider \eqref{eqn:diff} without the $d_{\xi_d}\Psi$ information. Then we arrive at
\begin{equation*}
\begin{gathered}
p(y_6,\partial_{y_6}\varphi)=0,\quad\quad\quad \kappa_2(y_6,\partial_{y_6}\varphi)=(y_4,\xi)\\
\kappa_2(y_2',0,-d_{y_2'}\Phi_2(y_4',y_{3,d},y_2',0,\theta_2),\zeta_d)=(y_4',y_{3,d},\xi),\\
\kappa_1(y_2',-d_{y_2'}\Phi_2(y_4',y_{3,d},y_2',0,\theta_2))=(y_1',\eta')
\end{gathered}
\end{equation*}
In particular, 
\begin{align*} 
0=p(\kappa_2^{-1}(y_4,\xi))&=p(\kappa_2^{-1}(y_4',y_{3,d},\xi)+p(\kappa_2^{-1}(y_4',y_{3,d},\xi))-p(\kappa_2^{-1}(y_4',y_{3,d},\xi))\\
&=\xi_d^2-\eta_1+p(\kappa_2^{-1}(y_4',y_{3,d},\xi))-p(\kappa_2^{-1}(y_4',y_{3,d},\xi))
\end{align*}
Now, since $\kappa_2$ is a symplectomorphism,
$$\partial_{y_d}p\composed \kappa_2^{-1}=H_p\partial_{\xi_d}p=\O{}(\partial_{\xi_d}p).$$
But, since $H$ is totally bicharacteristic (in particular, nowhere curved), on $H_p\partial_{\xi_d}p|_{x_d=0}=0$
and 
$$p(\kappa_2^{-1}(y_4',y_{3,d},\xi))-p(\kappa_2^{-1}(y_4',y_{3,d},\xi))=\O{}(\xi_d(y_{3,d}-y_{4,d})+(y_{3,d}-y_{4,d})^2).$$
Therefore, there exists $\gamma_i,\gamma\in C^\infty$ such that 
$$\xi_d^2-\eta_1=\left(\gamma (\xi_d+(y_{3,d}-y_{4,d}))d_{\xi_d}+\sum_{i}\gamma_id_i\right)\Psi$$
where the second term does not contain a term with $d_{\xi_d}$. Moreover, 
$$hD_{\xi_d}e^{\frac{i}{h}\Psi}=(y_{3,d}-y_{4,d})e^{\frac{i}{h}\Psi}.$$ 
Therefore, 
$$\frac{\gamma(\xi_dhD_{\xi_d}+(hD_{\xi_d})^2)+\sum_{i}\gamma_ih D_i}{\xi_d^2-\eta_1}e^{\frac{i}{h}\Psi}=e^{\frac{i}{h}\Psi}.$$ 
So, as in Lemma \ref{lem:microlocalize} integrating by parts in all variables except $\xi_d$ results in terms of the form $\O{}((h\tilde{h}^{-1})^N)$. 
Following the analysis there, when we integrate by parts twice in $\xi_d$ and divide by $\xi_d^2-\eta_1$ once, we lose $\tilde{h}^{-2}$. Similarly, integrating by parts in $\xi_d$ once, dividing by $\xi_d^2-\eta_1$ once, and multiplying by $\xi_d$ we lose $\tilde{h}^{-1}$. Therefore, all the integrand is $\O{}((h\tilde{h}^{-1})^N)$ for any $N$. Together with the same analysis as at the end of Lemma \ref{lem:microlocalize}, this implies the lemma
\end{proof}

\section{Almost orthogonality and the completion of the proof}
\label{sec:final}
We will need two lemmas on almost orthogonality to complete the proof.
\begin{lemma}
\label{lem:almostOrthog1}
Suppose that $\supp \chi \subset[1,4]$. Then for $j,k\geq 0$
\begin{equation}
\begin{gathered}
\left\|\chi\left(\frac{b(x',hD_{x'})}{2^jh}\right)\chi\left(\frac{b(x',hD_{x'})}{2^kh}\right)\right\|_{L^2(H)\to L^2(H)}\leq \begin{cases}C&|j-k|\leq 2\\
0&|j-k|>2
\end{cases}
\end{gathered}
\end{equation}
\end{lemma}
\begin{proof}
The proof follows from the functional calculus.\end{proof}
\begin{lemma}
\label{lem:almostOrthog2}
Suppose that $H_p$ is nowhere tangent to $H$ to infinite order or is totally bicharacteristic. Then there exists $\e>0$ small enough so that for $\psi \in \mc{S}$ with $\hat{\psi}\subset [-\e,\e]$, $\chi\in \Cc$ with $\supp \chi\subset [1,4]$, $j,k\geq 0$, and $A,B\in \Ph{1}{}(M)$
\begin{multline*}
\left\|\gamma_HB\chi\left(\frac{|\partial_\nu p|(x,hD)}{2^jh^{1/2}}\right)\psi\left(\frac{P}{h}\right)\chi\left(\frac{|\partial_\nu p|(x,hD)}{2^kh^{1/2}}\right)A\gamma_H^*\right\|_{L^2(H)\to L^2(H)}\\ 
\leq h^{-1/2-}2^{-(j+k)(1/2-)}
\begin{cases}C&|j-k|\leq 2\\
C_N2^{-(j+k)N}&|j-k|>2.
\end{cases}
\end{multline*}
\end{lemma}

To prove Lemma \ref{lem:almostOrthog2}, we need the following dynamical lemma.
\begin{lemma}
\label{lem:dynamics}
Suppose that $H=\{x_d=0\}$ is either nowhere tangent to $H_p$ to infinite order or totally bicharacteristic. Let $\Phi_t=\exp(tH_p)$ denote the Hamiltonian flow of $p$. For all $M>0$, there exists $\e>0$ small enough so that 
$$|t|\leq \e,\quad\quad 0<|\partial_\nu p(x',0,\xi)|\leq M,\quad\quad \Phi_t(x',0,\xi)\in T^*M|_{H}$$
implies 
$$c|\partial_\nu p(x',0,\xi)|\leq |\partial_\nu p\composed \Phi_t(x',0,\xi)|\leq C|\partial_\nu p(x',0,\xi)|.$$
\end{lemma}
\begin{remark}
One can see using the example 
$$H:=\left\{\begin{aligned}(x,e^{-1/x^2})&&x>0\\(x,0)&&x\leq 0\end{aligned}\right\}\subset\re^2$$
with $P=-h^2\Delta-1$ that if $H$ is tangent to $H_p$ to infinite order but is not totally bicharacteristic then the conclusion of Lemma \ref{lem:dynamics} may not hold.
\end{remark}
\begin{proof}
We may assume $t\geq 0$, the proof of the opposite case being identical.
Let 
$$(x(t),\xi(t)):=\exp(tH_p)(x',0,\xi).$$ Then 
$$ \dot{x}_d(t)=\partial_{\xi_d}p(x(t),\xi(t)).$$
First, observe that for any fixed $\delta>0$, if $|\dot{x}_d(0)|>\delta$, then there exists $\e$ small enough so that 
$|x_d(t)|>0$ for $0<t<\e$. Hence, the claim is trivial for $|\dot{x}_d(0)|>\delta$ for any fixed $\delta$. 

\noindent {\textbf{ $H$ is nowhere tangent to $H_p$ to infinite order.}}
\par  There exists $\delta_0>0$ small enough, $C>0$ large enough so that if $\delta<\delta_0$ and
$$x_d=0,\quad\quad 0<|\partial_{\xi_d}p(x_0',0,\xi_0)|<\delta_0$$
then there exists $(y',\eta)$ with $d((x_0',0,\xi_0),(y,\eta))<C|\partial_{\xi_d}p(x_0',0,\xi_0)|$ such that $\partial_{\xi_d}p(y',0,\eta)=0.$ Therefore, there exists $a>0$, $K>0$ such that if $x_d(0)=0$, and $|\dot{x}_d(t)|<\delta_0$,
then there exists $2\leq k\leq K$ and $|A|>a>0$ such that 
\begin{gather*} x_d(t)=\partial_{\xi_d}p(0)t+ \frac{A}{k}t^k+\O{}(t^{k+1})+\O{}(t^2\partial_{\xi_d}p(0))\\
\partial_{\xi_d}p(t)=\partial_{\xi_d}p(0)+At^{k-1}+\O{}(t^{k})+\O{}(t\partial_{\xi_d}p(0)).
\end{gather*}
Therefore, suppose $x_d(t)=0$, $t\neq 0$. Then 
$$\partial_{\xi_d}p(t)=\partial_{\xi_d}p(0)(1-k+\O{}(t))+\O{}(t^k).$$

Now, since $|t|<\e$, if $|t|^k>\frac{1}{4}|\partial_{\xi_d}p(0)|$ then $|t|^{k-1} >\frac{1}{4|t|}|\partial_{\xi_d}p(0)|$ and hence for $\e>0$ small enough,
$$|x_d(t)|\geq |t|\left(\left(\frac{|A|}{k}-C|t|\right)t^{k-1}-(1+C|t|)|\partial_{\xi_d}p(0)|\right)\geq \frac{a}{100K}|\partial_{\xi_d}p(0)|\geq 0.$$
Therefore, 
$$|\partial_{\xi_d}p(0)(|(1-k)|-1/2)
\leq |\partial_{\xi_d}p(t)|\leq |\partial_{\xi_d}p(0)|(|(1-k)|+1/2)$$
when $0<|\partial_{\xi_d}p(0)|<\delta_0.$ In particular, 
$$\frac{1}{2}|\partial_{\xi_d}p(0)|
\leq |\partial_{\xi_d}p(t)|\leq |\partial_{\xi_d}p(0)|2K$$

\noindent {\textbf{$H$ is totally bicharacteristic}}
\par When $H$ is totally bicharacteristic, 
$$\ddot x_d=H_p\partial_{\xi_d}p=\O{}(x_d^2+|\partial_{\xi_d}p|).$$
Now, if $x_d(0)=0$,
\begin{align*} 
|\dot{x}_d(t)-\dot{x}_d(0)|&=\left|\int_0^t\ddot{x}_d(s)ds\right|\leq C\int_0^tx_d^2(s)+|\dot{x}_d(s)|ds\\
&=C\int_0^t\left(\int_0^s\dot{x}_d(w)dw\right)^2+C\int_0^t|\dot{x}_d(s)|ds\\
&\leq C\int_0^t\int_0^s\dot{x}^2_d(w)sdwds+C\int_0^t|\dot{x}_d(s)|ds\\
&= C\frac{1}{2}\int_0^t\dot{x}^2_d(w)(t-w)^2dw+C\int_0^t|\dot{x}_d(s)|ds\\
&\leq C \int_0^t\frac{1}{2}t^2\dot{x}^2_d(s)+|\dot{x}_d(s)|ds
\end{align*}
Now, let 
$$t_0:=\inf_{t>0}\left\{|\dot{x}_d(t)-\dot{x}_d(0)|\geq\frac{|\dot{x}_d(0)|}{2}\right\}.$$
If $t_0>\e$ then we are done. If not, then $|\dot{x}_d(t_0)-\dot{x}_d(0)|=\frac{|\dot{x}_d(0)|}{2}.$
Therefore, 
\begin{align*}
|\dot{x}_d(t_0)-\dot{x}_d(0)|= \frac{1}{2}|\dot{x}_d(0)|&\leq\frac{3}{2}C \int_0^{t_0}\frac{3}{4}t_0^2\dot{x}^2_d(0)+|\dot{x}_d(0)|ds\\
&=C\frac{3}{2}\left(\frac{1}{4}t_0^3|\dot{x}_d(0)|+t_0\right)|\dot{x}_d(0)|
\end{align*}
So, 
$$C^{-1}\leq \frac{3}{4}t_0^3|\dot{x}_d(0)|+3t_0$$
and choosing $\e>0$ small enough finishes the proof of the lemma.
\end{proof}
\begin{proof}[Proof of Lemma \ref{lem:almostOrthog2}]
We assume that $H$ is given by $\{x_d=0\}$. Then the kernel of
$$\gamma_H\chi\left(\frac{|\partial_\nu p|(x,hD)}{2^jh^{1/2}}\right)\psi\left(\frac{P}{h}\right)\chi\left(\frac{|\partial_\nu p|(x,hD)}{2^kh^{1/2}}\right)\gamma_H^*$$
is given microlocally near a point $(x_0,\xi_0)\in T^*H$ by
\begin{multline*} 
Ch^{-3d}\int e^{\frac{i}{h}(\la x'-w_1',\eta'\ra -w_{1,d}\eta_d+\varphi(t,w_1,\theta)-\la w_2,\theta\ra +w_{2,d}\xi_d+\la w_2'-y',\xi'\ra)}
\widehat{\psi}(t)\tilde{a}dw_1dw_2dtd\theta d\xi d\eta\end{multline*}
where 
$$\supp \tilde{a}\subset\{\partial_\nu p(x',0,\eta)2^{-k}h^{-1/2}\in \supp \chi,\,\partial_\nu p(w_2,\xi)2^{-j}h^{-1/2}\in \supp\chi\}$$
and 
$$|\partial^\alpha \tilde{a}|\leq C_\alpha 2^{-\min(j,k)|\alpha|}h^{-|\alpha|/2}.$$
Let $\Phi_t$ denote the Hamiltonian flow of $p$. Then the phase is stationary when 
$$p(x,\xi)=0,\quad\quad \Phi_t(y',0,\xi)=(x',0,\xi)$$ 
and by Lemma \ref{lem:dynamics} there exist $\e,\delta>0$ so that for $|t|\leq \e$, and $0<|\partial_\nu p(x',0,\xi)|\leq \delta$, there exist $c,C>0$ so that when $\Phi_t(y',0,\xi)\in T^*M|_H$, 
$$c|\partial_\nu p|\leq |\partial_\nu p\composed \Phi_t|\leq C|\partial_\nu p|.$$
In particular, on the support of the integrand, 
$$|\Phi_t(y',0,\xi)-(x',0,\xi)|>2^{\max(k,j)}h^{1/2}.$$
Therefore, integration by parts proves the estimate with $h^{-1/2-}2^{-(1/2-)(j+k)}$ replaced by $h^{-M}$. To obtain the lemma, we then repeat the argument at the end of Lemma \ref{lem:microlocalize} using the fact that 
$$\left\|\gamma_H\chi\left(\frac{|\partial_\nu p|(x,hD)}{2^jh^{1/2}}\right)\psi\left(\frac{P}{h}\right)\right\|_{L^2\to L^2}\leq Ch^{-1/4}2^{-j/2}.$$
\end{proof}

Now, let $\rho\leq 1$, $s\in\re$, $\mu\in \re$, $G_1^{\rho,s}$ be as in the introduction, $A\in \Ph{1}{}(M)$ have 
$$|\sigma(A)|\leq C|\partial_\nu p|^\mu,\quad\quad\text{ on }|p|\leq 1,$$
 and $\chi \in \Cc(\re)$ so that 
$$\sum_{j=0}^\infty\chi(2^{-j}x)\equiv 1\quad\quad\text{ for }x\in[1/8,\infty).$$
Define
$$\chi_0=1-\sum_{j=0}^\infty \chi.$$
Let $Mh^{-1}\geq 2^J\gg h^{-1}$.
Then, since $\partial_\nu p$ is uniformly bounded above on $\{p=0\}$, 
$$A\psi\left(\frac{P}{h}\right)=\left(\chi_0\left(\frac{|\partial_\nu p(x,hD_x)|}{2^jh^{\rho/2}}\right)+\sum_{j=0}^J\chi\left(\frac{|\partial_\nu p(x,hD_x)|}{2^jh^{\rho/2}}\right)\right)A\psi\left(\frac{P}{h}\right).$$
Define
$$T_j=G_1^{\rho,s}(b(x',hD_{x'}))\gamma_H\chi\left(\frac{|\partial_\nu p(x,hD_x)|}{2^jh^{\rho/2}}\right)A\psi\left(\frac{P}{h}\right)$$
so that 
$$G_1^{\rho,s}(b(x',hD_{x'}))\gamma_HA\psi\left(\frac{P}{h}\right)=\sum_{j=0}^JT_j+\begin{cases}\O{L^2\to L^2}(h^\infty)&\rho<1\\
\O{L^2\to L^2}(h^{s+\mu/2})&\rho=1.\end{cases}$$
The $h^\infty$ error is clearly negligible. To see that the $h^{s+\mu/2}$ error is negligible, we use the estimate \eqref{e:stdEst}
$$\left\|\gamma_H A\psi\left(\frac{P}{h}\right)\right\|_{L^2(M)\to L^2(H)}\leq Ch^{-1/4}$$
 to see that for $s\geq1/4-\mu/2$, the error term is uniformly bounded in $h$. 

Now, by Lemmas \ref{lem:microlocalize} and \ref{lem:microlocalizeTotallyBicha},
\begin{multline*}T_j=G_1^{\rho,s}(b(x',hD_{x'}))\chi_j\left(\frac{b(x',hD_{x'})}{2^jh^{\rho}}\right)\gamma_H\chi_\nu\left(\frac{\partial_\nu p(x,hD)}{2^{j/2}h^{\rho/2}}\right)A\psi\left(\frac{P}{h}\right)\\
+\O{L^2\to L^2}(h^{-}h^{\rho(s-1/4+\mu/2)}2^{j(s-1/4+\mu/2)}(h^{1-\alpha\rho}2^{-\alpha j})^\infty)\end{multline*}
where $\alpha=3/2$ unless $H$ is totally bicharacteristic, in which case $\alpha=1$.
So, taking 
$\rho<\alpha^{-1}$ or $s> 1/4-\mu/2$, the remainder term is summable with sum bounded uniformly in $h$. Therefore, we may analyze only 
$$\tilde{T}_j=G_1^{\rho,s}(b(x',hD_{x'}))\chi_j\left(\frac{b(x',hD_{x'})}{2^jh^{\rho}}\right)\gamma_H\chi_\nu\left(\frac{\partial_\nu p(x,hD)}{2^{j/2}h^{\rho/2}}\right)A\psi\left(\frac{P}{h}\right).$$

In this case, Lemmas \ref{lem:normal}, \ref{lem:almostOrthog1}, and \ref{lem:almostOrthog2} show that if $H_p$ is nowhere tangent to $H$ to infinite order or totally bicharacteristic then for $\e>0$ small enough, and $|j-k|>2$,
\begin{equation*} 
\|\tilde{T}_j\tilde{T}_k^*\|_{L^2\to L^2}+\|\tilde{T}_j^*\tilde{T}_k\|_{L^2\to L^2}
\leq    C_N2^{(j+k)(s+\mu/2-1/4-)}h^{\rho(2 s+\mu-1/2-)} h^{(1-\rho)N}2^{-(j+k)N}
\end{equation*}
Moreover, 
$$\|\tilde{T}_j\|^2_{L^2\to L^2}\leq C2^{j(2s+\mu-1/2)}h^{\rho(2s+\mu-1/2)}.$$
Then by the Cotlar--Knapp--Stein lemma (see for example \cite[Theorem C.5]{EZB}) if $s\geq 1/4-\mu/2$ and $\rho<1$ or $s>1/4-\mu/2$ and $\rho\leq 1$ 
$$\|\sum_{j}\tilde{T}_j\|_{L^2\to L^2}\leq C.$$

If $H_p$ is somewhere tangent to infinite order, then Lemma \ref{lem:almostOrthog2} does not apply and therefore we instead estimate for $\rho<1$,
$$\left\la \sum_{j}\tilde{T}_{j}u,\sum_k\tilde{T}_ku\right\ra \leq C\|u\|^2_{L^2(M)}\begin{cases}
\log h^{-1}&s=1/4-\mu/2\\
1&s>1/4-\mu/2
\end{cases}.$$
Hence, 
\begin{lemma}
Let $\alpha=1$ if $H$ is totally bicharacteristic, and $3/2$ otherwise, $\mu \geq 0$, $0\leq\rho\leq \alpha^{-1}$ and $s\geq 1/4-\mu/2$, and $A\in \Ph{1}{}(M)$ have 
$$|\sigma(A)|\leq C|\partial_\nu p(x,\xi)|^\mu,\quad\quad \text{on }|p|\leq 1.$$
Then if $\rho < \alpha^{-1}$ or $s>1/4-\mu/2$, and either $H_p$ is nowhere tangent to infinite order to $H$ or $H$ is totally bicharacteristic,
$$\left\|G_1^{\rho,s}(b(x',hD_{x'}))\gamma_HA\psi\left(\frac{P}{h}\right)\right\|_{L^2(M)\to L^2(H)}\leq C.$$
If instead $H_p$ is somewhere tangent to infinite order to $H$ but not totally bicharacteristic, then 
$$\left\|G_1^{\rho,s}(b(x',hD_{x'}))\gamma_HA\psi\left(\frac{P}{h}\right)\right\|_{L^2(M)\to L^2(H)}\leq C\begin{cases}(\log h^{-1})^{1/2}&s=1/4-\mu/2\\
1&s>1/4-\mu/2\end{cases}.$$
\end{lemma}

\section{Optimality of the power $1/4$.}
\label{sec:optimal}
We will show that the power $1/4$ is optimal for curved $H$ and for totally geodesic $H$. In fact, we will show that the power $1/4$ is sharp at every scale. More precisely, letting $\mu=2/3$ if $H$ is curved, and $1$ if it is totally geodesic, for all $0\leq \rho_1<\rho_2< \mu$, we give examples of eigenfunctions $u_h$ with 
$$\|(G_{1}^{\rho_2,s}-G_{1}^{\rho_1,s})(1+h^2\Delta_H))u_h|_H\|_{L^2(H)}\geq ch^{\rho_2(s-1/4)}.$$
\subsection{$H$ curved}
Consider the unit disk $B(0,1)\subset \re^2$. Then the Dirichlet eigenfunctions are given by 
$$u=c_n J_n(\lambda r)e^{in\theta},\quad\quad J_n(\lambda)=0.$$
Let $H=\{|x|=1/2\}$ and fix $\alpha<2/3$. 
By the uniform asymptotics for zeroes of Bessel functions (\cite[Section 10.20,10.21]{NIST} ), the $m^{\text{th}}$ zero of the $n^\text{th}$ Bessel function is given by
$$j_{n,m}=nz(\zeta)+\O{}(n^{-1}),\quad\quad \zeta=n^{-2/3}a_m$$
where $a_m$ is the $m^{\text{th}}$ zero of the Airy function and $\zeta$ solves
$$\left(\frac{d \zeta}{dz }\right)^2=\frac{1-z^2}{\zeta z^2}$$
and is infinitely differentiable on $0<z<\infty.$ 
The zeroes of the airy function have 
$$a_m=-\left(\frac{3}{8}\pi(4m-1)\right)^{2/3}+\O{}(m^{-4/3}).$$ 

Now, since $z(0)=1$ and $\lim_{\zeta\to -\infty}z(\zeta)=\infty$, there exists $\zeta_0<0$ with $z(\zeta_0)=2$ and, moreover, there exists $c>0$ so that $\zeta_0<-c<0$. Hence,  for any $\beta<1$, $M>0$ there exists $m\geq cn$,  such that 
$$n^{-2/3} a_m\in \zeta_0 -[Mn^{-\beta},(M+1)n^{-\beta}].$$ 
In particular,  
$$z(n^{-2/3}a_m)\in 2+[Mn^{-\alpha},(M+1)n^{-\alpha}]$$
which implies there exists 
\begin{equation}\label{eqn:lambPos}\lambda_n\in 2n+[Mn^{1-\alpha},(M+1)n^{1-\alpha}]
\end{equation} such that 
$$u_n=c_ne^{in\theta}J_n(\lambda(n)r)$$
is a Dirichlet eigenfunction. One can see that to $L^2$ normalize $u_n$, $c_n\sim cn^{1/2}.$ Moreover, 
$$J_n(nz)=\left(\frac{4\zeta}{1-z^2}\right)^{1/4}\left(\frac{Ai(n^{2/3}\zeta)}{n^{1/3}}+\O{}(n^{-2/3})\right).$$
So, evaluating at $\frac{1}{2}\lambda_n$ and using the asymptotics for the Airy function gives, since $0\leq \alpha<2/3$, 
$$\left|J_n\left(\frac{1}{2}\lambda_n\right)\right|\geq cn^{-1/3-(2/3-\alpha)/4}.$$
Therefore, 
\begin{equation}\label{eqn:curvedAsymptotic}u_n|_{H}=A_ne^{in\theta},\quad\quad\quad |A_n|\geq cn^{-1/3-(2/3-\alpha)/4},\end{equation}
and, taking $M$ large enough in \eqref{eqn:lambPos}, $\rho_1<\alpha\leq \rho_2$, 
\begin{equation}
\label{eqn:estCurved}
\begin{gathered}G_{1}^{\rho_2,s}(1+\lambda_n^{-2}\Delta_H)u_n|_H= \left(1-4\frac{n^2}{\lambda_n^2} \right)^su_n|_H\sim c n^{-\alpha s}u_n|_H\\
G_{1}^{\rho_1,s}(1+\lambda_n^{-2}\Delta_H)u_n|_H=0
\end{gathered}
\end{equation}
Hence, for $\alpha<2/3$ and $\rho_1<\alpha<\rho_2$, using \eqref{eqn:curvedAsymptotic} and \eqref{eqn:estCurved},
$$\|(G_{1}^{\rho_2,s}-G_{1}^{\rho_1,s})(1+\lambda_n^{-2}\Delta_H)u_n|_H\|_{L^2(H)}\geq cn^{\alpha(1/4-s)}\geq \tilde{c}\lambda_n^{\alpha(1/4-s)}$$
for some $\tilde{c}>0$.
In particular, for $s<1/4$, this unbounded.

\subsection{$H$ totally geodesic}
Consider the unit sphere, $S^2\subset\re^3$. Let 
$$[0,2\pi)\times [0,\pi]\ni (\theta,\phi)\mapsto (\cos \theta\sin\phi,\sin \theta\sin\phi,\cos \phi)\in S^2$$
be coordinates on $S^2$. Then an orthonormal basis of Laplace eigenfunctions is given by 
$$Y^m_l(\theta,\phi)=\left(\frac{(l-m)!(2l+1)}{4\pi(l+m)!}\right)^{1/2}e^{im\theta}P^m_l(\cos \phi),\quad\quad -l\leq m\leq l,$$
where $P^m_l$ is an associated Legendre function (see for example \cite[Section 14.30]{NIST}). For the definition of $P^m_l$ see \cite[Section 14.2]{NIST}. Note that 
$$(-\Delta_{S^2}-\lambda_l^2)Y^m_l=0,\quad\quad \lambda_l:=\sqrt{l(l+1)}.$$

Let $H:=\{\phi=\frac{1}{2}\pi\},$ fix $\alpha<1$ and let 
\begin{equation}
\label{eqn:mpos}m_l\in l-[M,M+1]l^{1-\alpha}
\end{equation} so that $l+m_l\in 2\mathbb{Z}$ (i.e. is even). Then by \cite[Section 14.30ii]{NIST}, 
$$Y^{m_l}_l\left(\theta,\frac{1}{2}\pi\right)=\frac{(-1)^{(l+m_l)/2}}{2^l\left(\frac{1}{2}(l-m_l)\right)!\left(\frac{1}{2}(l+m_l)\right)!}\left(\frac{(l-m_l)!(l+m_l)!(2l+1)}{4\pi}\right)^{1/2}e^{im_l\theta}=:A_{l}e^{im_l\theta}.$$
After some straightforward, but tedious computations, one finds that for some $a>0$, 
$|A_{l}|\geq cl^{\alpha/4}.$
Hence, 
\begin{equation}
\label{eqn:flatAsymptotic}Y^{m_l}_l|_H=A_le^{im_l\theta},\quad\quad \quad|A_{l}|\geq cl^{\alpha/4}.
\end{equation}

Now, for $\rho_1<\alpha\leq \rho_2$, using \eqref{eqn:mpos} and taking $M$ large enough,
\begin{equation}
\label{eqn:estFlat}
\begin{gathered} G^{\rho_2,s}_{1}(1+\lambda_l^{-2}\Delta_{H})Y^{m_l}_l|_{H}\geq \left(1-\frac{m_l^2}{\lambda_l^2}\right)^sY^{m_l}_l|_H\\
G^{\rho_1,s}_{1}(1+\lambda_l^{-2}\Delta_{H})Y^{m_l}_l|_{H}=0.
\end{gathered}
\end{equation}
Hence, for $\alpha<1$ and $\rho_1<\alpha<\rho_2$, using \eqref{eqn:flatAsymptotic} and \eqref{eqn:estFlat},
$$\|(G_{1}^{\rho_2,s}-G_{1}^{\rho_1,s})(1+\lambda_l^{-2}\Delta_H)Y^{m_l}_l|_H\|_{L^2(H)}\geq a_1l^{\alpha(1/4-s)}\geq \tilde{a}\lambda_l^{\alpha(1/4-s)}$$
for some $\tilde{a}>0$.
In particular, for $s<1/4$, this is unbounded.

\section{Application to Cauchy Quantum Ergodic Restrictions}
\label{sec:QER}

We now prove Theorem \ref{thm:QER}. Let $u_h$ be a quantum ergodic sequence of Laplace eigenfunctions. Then by the quantum ergodic restriction theorem for Cauchy data \cite{CTZ} (see also \eqref{e:cauchy}), for $A\in \Ph{}{}(H)$,
$$\la Ah\partial_\nu u|_H,h\partial_\nu u|_H\ra +\la (1+h^2\Delta_H)Au|_H,u|_H\ra \to \frac{4}{\mu_L(S^*M)}\int_{B^*H}\sigma(A)\sqrt{1-|\xi'|_g^2}dxd\xi'.$$
Let $\chi_\delta \in \Cc(\re)$ with $\chi_\delta\equiv 1$ on $\{|x|\leq 1-2\delta\}$, with $\supp \chi\subset \{|x|\leq 1-\delta\}$ and $\psi_\delta=1-\chi_\delta$. Then denote 
$$\chi_\delta(h):=\chi_\delta(-h^2\Delta_H),\qquad\psi_\delta(h):=\psi_\delta(-h^2\Delta_H),\quad\quad G_1^{\rho,s}(h):=G_1^{\rho,s}(1+h^2\Delta_H).$$
Let $A\in \Ph{}{}(H)$, $s<1/2$, and consider
\begin{align} 
&\la G_1^{2/3,-s}(h)Ah\partial_\nu u,h\partial_\nu u\ra +\la (1+h^2\Delta_H)G_1^{2/3,-s}(h)A u|_H,u|_H\ra\nonumber\\
&\quad=\la (\chi_\delta(h)+\psi_\delta(h))G_1^{2/3,-s}(h)Ah\partial_\nu u,h\partial_\nu u\ra
+\la (\chi_\delta(h)+\psi_\delta(h))G_1^{2/3,1-s}(h) Au|_H,u|_H\ra\nonumber\\
&\quad=\frac{4}{\mu_L(S^*M)}\int \chi_\delta(|\xi'|_g^2)\sigma(A)(x,\xi)(1-|\xi'|_g^2)^{1/2-s}dxd\xi' +\o{\delta}(1)\nonumber\\
&\quad\qquad+\la \psi_\delta(h)G_1^{2/3,-s}(h)Ah\partial_\nu u,h\partial_\nu u\ra +\la \psi_\delta(h)G_1^{2/3,1-s}(h)A u|_H,u|_H\ra\label{e:yank}
\end{align} 
The proof of Theorem \ref{thm:QER} will be complete after we estimate the term in \eqref{e:yank}.

Let $\tilde{G}_1^{2/3,\beta}$ be defined as in \eqref{e:Gdef} but with $\tilde{\chi}_1$ replacing $\chi_1$ so that $\tilde{\chi}_1 \chi_1=\chi_1$ and $\supp \tilde{\chi}\subset [1/2,\infty)$. Next, let $\tilde{\psi}_\delta$ have $\supp \tilde{\psi}_\delta\subset \{|x|\geq1-3\delta\}$ and $\tilde{\psi}_\delta\psi_\delta=\psi_\delta.$ Then,
$$\psi_\delta(h)G_1^{2/3,\beta_1+\beta_2}(h)=\psi_\delta(h)G_1^{2/3,\beta_1}(h)\tilde\psi_\delta(h)\tilde G_1^{2/3,\beta_2}(h)$$
and hence
\begin{align*}
\psi_\delta(h)G_1^{2/3,\beta_1+\beta_2}(h)A&=\psi_\delta(h)G_1^{2/3,\beta_1}(h)A\tilde{\psi}_\delta(h)\tilde{G}_1^{2/3,\beta_2}(h)+\psi_\delta(h)G_1^{2/3,\beta_1}(h)[\tilde{\psi}_\delta(h)\tilde{G}_1^{2/3,\beta_2}(h),A]\\
&=\psi_\delta(h)G_1^{2/3,\beta_1}(h)A\tilde{\psi}_\delta(h)\tilde{G}_1^{2/3,\beta_2}(h)+\O{L^2\to L^2}(h^{1-2/3(1-\min(\beta_1,0)-\min(\beta_2,0))}).
\end{align*}
Therefore, 
\begin{align*}\la \psi_\delta(h)G_1^{2/3,1-s}(h)Au|_H,u|_H\ra&=\la \psi_\delta(h)G_1^{2/3,(1-s)/2}(h)Au|_H,\tilde{\psi}_\delta(h)\tilde{G}_1^{2/3,(1-s)/2}(h)u|_H\ra\\
&=\la \psi_\delta(h)G_1^{2/3,0}(h)A\tilde{\psi}_\delta(h)\tilde G_1^{2/3,(1-s)/2}(h)u|_H,\tilde{\psi}_\delta(h)\tilde{G}_1^{2/3,(1-s)/2}(h)u|_H\ra\\
&\qquad+\la \O{L^2\to L^2}(h^{1-2/3})u|_H,\tilde{\psi}_\delta(h)\tilde{G}_1^{2/3,(1-s)/2}(h)u|_H\ra
\end{align*}
\begin{align*}
&\la \psi_\delta(h)G_1^{2/3,-s}(h)Ah\partial_\nu u|_H,h\partial_\nu u|_H\ra\\
&\qquad\quad=\la \psi_\delta(h)G_1^{2/3,-s/2}(h)Ah\partial_\nu u|_H,\tilde{\psi}_\delta(h)\tilde{G}_1^{2/3,-s/2}(h)h\partial_\nu u|_H\ra\\
&\qquad\quad=\la \psi_\delta(h)G_1^{2/3,0}(h)A\tilde{\psi}_\delta(h)\tilde G_1^{2/3,-s/2}(h)h\partial_\nu u|_H,\tilde{\psi}_\delta(h)\tilde{G}_1^{2/3,-s/2}(h)h\partial_\nu u|_H\ra\\
&\qquad\qquad\qquad+\la \O{L^2\to L^2}(h^{1-2/3(1+s)})h\partial_\nu u|_H,\tilde{\psi}_\delta(h)\tilde{G}_1^{2/3,-s/2}(h)h\partial_\nu u|_H\ra
\end{align*}
Now, by the functional calculus of self adjoint operators,
$$\tilde{\psi}_\delta(h)\tilde{G}_1^{2/3,\beta}(h)=Z_\alpha\tilde{\psi}_\delta(h)\tilde{G}_1^{2/3,\beta-\alpha}(h),\qquad Z_\alpha=\O{L^2\to L^2}(\delta^\alpha).$$
By Theorem \ref{thm:general}, for $\beta<1/4$,
$$\|\tilde G_1^{2/3, 1/2-\beta}(h)u|_H\|_{L^2(H)}+\|\tilde G_1^{2/3,-\beta}(h)h\partial_\nu u|_H\|_{L^2(H)}\leq C.$$
Hence, we have
$$\|\tilde{\psi}_\delta(h)G_1^{2/3,(1-s)/2}(h)u\|_{L^2(H)}+\|\tilde{\psi}_\delta(h)G_1^{2/3,-s/2}(h)h\partial_\nu u\|_{L^2(H)}\leq C\delta^{s/2-1/4-}.$$
In particular, using \eqref{e:stdEst} and \eqref{eqn:neumannBound} to see that
$\|u|_H\|_{L^2(H)}\leq Ch^{-1/4}$ and $\|h\partial_\nu u\|_{L^2(H)}\leq C,$
together with $h^{3/4-2/3}=\o{}(1)$ and $h^{1-2/3(1+s)}=\o{}(1)$, 
$$\la \psi_\delta(h)G_1^{2/3,-s}(h)Ah\partial_\nu u,h\partial_\nu u\ra +\la \psi_\delta(h)G_1^{2/3,1-s}(h)A u|_H,u|_H\ra=\o{\delta}(1)+\O{}(\delta^{1/2-s-}).$$
Therefore, 
\begin{multline*} 
\la G_1^{2/3,-s}(h)Ah\partial_\nu u,h\partial_\nu u\ra +\la (1+h^2\Delta_H)G_1^{2/3,-s}(h)A u|_H,u|_H\ra\\=\frac{4}{\mu_L(S^*M)}\int \chi_\delta(|\xi'|_g^2)\sigma(A)(x,\xi)(1-|\xi'|_g^{2})^{1/2-s}dxd\xi'+\o{\delta}(1)+\O{}(\delta^{1/2-s-}).
\end{multline*}
So, since $s<1/2$, letting $h\to 0$ and then $\delta\to 0$, we have 
\begin{multline*}
\la G_1^{2/3,-s}(1+h^2\Delta_H)Ah\partial_\nu u,h\partial_\nu u\ra +\la G_1^{2/3,1-s}(1+h^2\Delta_H) u|_H,u|_H\ra\\
\to \frac{4}{\mu_L(S^*M)}\int_{B^*H}\sigma(A)(x,\xi)(1-|\xi'|_g^2)^{1/2-s}dxd\xi',
\end{multline*}
completing the proof of Theorem \ref{thm:QER}.

\bibliography{biblio.bib}
\bibliographystyle{alpha}
\end{document}